\documentclass{amsart}

\usepackage{mathrsfs}
\usepackage{amsfonts}
\usepackage{indentfirst,latexsym,bm}
\usepackage{amssymb, amsthm}
\usepackage{amsopn}
\usepackage{epsfig}
\usepackage{latexsym}
\usepackage{graphicx}

\newtheorem{theorem}{Theorem}[section]
\newtheorem{lemma}[theorem]{Lemma}
\newtheorem{corollary}[theorem]{Corollary}
\newtheorem{proposition}[theorem]{Proposition}

\title[Intersections of homogeneous Cantor sets and
beta-expansions]{\bf Intersections of homogeneous Cantor sets and
beta-expansions }

\author[D. Kong]{Derong Kong}
\author[W. Li]{Wenxia Li}
\author[M. Dekking]{Michel Dekking}

\address{3TU Applied Mathematics Institute and Delft
University of Technology, Faculty EWI, P.O.~Box 5031, 2600 GA
Delft, The Netherlands.}

\address{Wenxia Li.  Department of Mathematics, East China Normal University, Shanghai 200241,
People's Republic of China}

\email{D.Kong@tudelft.nl,\quad wxli@math.ecnu.edu.cn,\quad F.M.Dekking@tudelft.nl}

\thanks{The second author is supported by the National Natural Science
Foundation of China no 10971069,  Shanghai Leading Academic
Discipline Project no B407 and Shanghai Education Committee Project no 11ZZ41.}

\date{\today}

\newcommand{\N}{\ensuremath{\infty}}
\newcommand{\ep}{\varepsilon}
\newcommand{\pz}{\pi_{N}}
\newcommand{\pf}{\pi_{\pm N}}
\newcommand{\ps}{\pi_{2N-1}}
\newcommand{\ON}{\Omega_{N}}
\newcommand{\OPN}{\Omega_{\pm N}}
\newcommand{\OTN}{\Omega_{2N-1}}
\newcommand{\UPN}{\mathcal{U}_{\beta,\pm N}}
\newcommand{\UTN}{\mathcal{U}_{\beta,2N-1}}
\newcommand{\SPN}{\mathcal{S}_{\beta,\pm N}}
\newcommand{\STN}{\mathcal{S}_{\beta,2N-1}}

\begin{document}

\begin{abstract}
 Let $\Gamma_{\beta,N}$ be the $N$-part homogeneous
Cantor set with $\beta\in(1/(2N-1),1/N)$. Any string $(j_\ell)_{\ell=1}^\N$ with $j_\ell\in\{0,\pm 1,\dots,\pm(N-1)\}$
 such that $t=\sum_{\ell=1}^\N j_\ell\beta^{\ell-1}(1-\beta)/(N-1)$ is called a code of $t$. Let $\UPN$ be
 the set of $t\in[-1,1]$ having a unique code, and let $\SPN$ be the set of $t\in\UPN$ which
 make the intersection $\Gamma_{\beta,N}\cap(\Gamma_{\beta,N}+t)$ a self-similar set. We characterize the set $\UPN$ in a geometrical
 and algebraical way, and give a sufficient and necessary condition for $t\in\SPN$.
 Using techniques from beta-expansions, we show that there is a critical point $\beta_c\in(1/(2N-1),1/N)$,
 which is a transcendental number, such that $\UPN$ has positive Hausdorff dimension if $\beta\in(1/(2N-1),\beta_c)$,
 and contains countably infinite many elements if
 $\beta\in(\beta_c,1/N)$. Moreover, there exists a second critical point $\alpha_c=\big[N+1-\sqrt{(N-1)(N+3)}\,\big]/2\in(1/(2N-1),\beta_c)$
  such that $\SPN$ has positive
 Hausdorff dimension if $\beta\in(1/(2N-1),\alpha_c)$, and contains countably infinite many
 elements if $\beta\in[\alpha_c,1/N)$.

\medskip


\medskip

\end{abstract}
\maketitle

\section{ Introduction}

Let $\{f_i(x)=r_i x+b_i\}_{i=1}^p$ be a family of
 functions on $\mathbb R$ with $0<|r_i|<1$. It is well known (cf.~\cite{Falconer}) that there exists a unique
 nonempty compact set $\Gamma \subseteq \mathbb R$ such that
 \begin{equation*}
 \Gamma=\bigcup_{i=1}^p f_i(\Gamma) .
 \end{equation*}
 In this case, $\Gamma$ is called the {\it self-similar set} generated by
the  {\it iterated function system} (IFS) $\{f_i(\cdot)\}_{i=1}^p$.

We will be interested in the self-similar set $\Gamma_{\beta,\Omega}$ generated by an IFS $\{\phi_d(\cdot): d\in\Omega\}$, where $\Omega$ is a finite set of integers, and
\begin{equation*}
 \phi_d(x)=\beta x+d(1-\beta)/(N-1), \quad x\in\mathbb R
\end{equation*}
for some $N\ge 2$ and $\beta\in(0,1/N)$.
It is well known that one can establish a surjective map
$\pi_{\Omega}:\Omega^\N\rightarrow\Gamma_{\beta, \Omega }$ by
letting
\begin{equation}\label{eq:pi}
\pi_\Omega(J)=\sum_{\ell=1}^\infty\frac{j_\ell\beta^{\ell-1}(1-\beta)}{N-1}
\end{equation}
for $J=(j_\ell)_{\ell=1}^\infty\in\Omega^\N$. The infinite string
$J$
 is called an \emph{$\Omega$-code} of $\pi_\Omega(J)$. Note that an element $x\in
 \Gamma _{\beta, \Omega }$ may have multiple $\Omega$-codes.
These $\Omega$-codes are closely related to the classical beta-expansions
 (cf.~\cite{Erdos,Glendinning Sidorov,Komornik  Loreti,Parry,Renyi,Sidorov,Vries Komornik}).
A sequence $(s_\ell)_{\ell=1}^\N\in\Omega^\N$ is called a
   \emph{$\beta$-expansion of $x$ with digit set $\Omega$}  if we can write
\begin{equation*}
     x=\sum_{\ell=1}^\infty s_\ell\beta^\ell,\quad s_\ell\in\Omega.
\end{equation*}

 Let $\ON:=\{0,1,\dots,N-1\}$. We simplify the notation $\Gamma
_{\beta , \ON }$ to $\Gamma _{\beta , N}$, so this set satisfies
\begin{equation*}
\Gamma _{\beta , N}=\bigcup _{d\in\ON}\phi _d(\Gamma _{\beta , N}).
\end{equation*}
 The set $\Gamma _{\beta , N}$ is
called the {\it $N$-part homogeneous Cantor set}. Thus
$\Gamma_{1/3,2}$ is the \emph{classical
 middle-third Cantor set} and $\Gamma_{\beta , 2}$ is the
 middle-$\alpha $ Cantor set with $\alpha =1-2\beta $.

 In terms of (\ref{eq:pi}), let $\pz:=\pi_{\ON}$. Thus we can rewrite
$\Gamma_{\beta,N}$ as
\begin{equation}\label{eq:expansion of Gamma}
\Gamma_{\beta, N}=\pz
\big(\ON^\N\big)=\left\{\sum_{\ell=1}^\infty\frac{j_\ell\beta^{\ell-1}(1-\beta)}{N-1}:j_\ell\in\ON\right\}.
\end{equation}
We consider the intersection of $\Gamma _{\beta , N}$ with its
translation by $t$. It is easy to check that
\begin{equation*}
\Gamma_{\beta,N}\cap(\Gamma_{\beta,N}+t)\ne\emptyset \quad\mbox{\rm
if and only if}\quad t\in \Gamma_{\beta,N}-\Gamma_{\beta,N}.
\end{equation*}
Here we denote for a real number $a$, and sets
$A,B\subseteq \mathbb R$, $aA:=\{ax: x\in A\}$, $A+B:=\{x+y: x\in A,
y\in B\}$, and $A+a:=A+\{a\}$.

It follows from Equation
(\ref{eq:expansion of Gamma}) that the difference set
$\Gamma_{\beta,N}-\Gamma_{\beta,N}$ can be written as
\begin{equation*}
\Gamma_{\beta,N}-\Gamma_{\beta,N}=\left\{\sum_{k=1}^\infty\frac{t_\ell\beta^{\ell-1}(1-\beta)}{N-1}:t_\ell\in\OPN\right\}
=\pf\big(\OPN^\infty\big)=\Gamma_{\beta,\OPN},
\end{equation*}
where $\OPN:=\ON-\ON=\{0,\pm 1,\dots,\pm(N-1)\}$ and
$\pf:=\pi_{\OPN}$. Since $\Omega_{2
N-1}=\{0,1,\dots,2N-2\}=\Omega_{\pm N}+N-1$, it is easy to see that
$ (t_\ell)_{\ell=1}^\infty$ is a $\Omega_{\pm N}$-code of
$t\in\Gamma_{\beta,N}-\Gamma_{\beta,N}$
 if and only if
$(t_\ell+N-1)_{\ell=1}^\infty$ is an $\beta$-expansion of
$(t+1)\beta(N-1)/(1-\beta)$ with digit set $\Omega_{2N-1}.$ Thus
some results and techniques from beta-expansions can be used to deal
with the difference set $\Gamma_{\beta,N}-\Gamma_{\beta,N}$.

In the past two decades, intersections of Cantor sets  have been
studied by several authors (cf.~\cite{Davis, Kenyon, Kraft1, Kraft2,
Kraft3,LiW2}). Recently, Deng et al.~\cite{Deng} gave
 a necessary and sufficient condition for  $t\in [-1,1]$ such that
  $\Gamma_{1/3,2}\cap(\Gamma_{1/3,2}+t)$ is a
 self-similar set. Their results were extended to the case
 $\Gamma_{\beta ,N}\cap(\Gamma_{\beta , N}+t)$ with
 $\beta\in(0,1/(2N-1)]$ by Li et al.~\cite{LiW1}, and to the
 case $\Gamma_{\beta,2}\cap (\Gamma_{\beta,2}+t)$ with
 $\beta\in(1/3,1/2)$ and $t$ having a unique $\Omega_{\pm 2}$-code by
 Zou et al.~\cite{Zou2}.

In this paper we consider arbitrary $N\ge 2$, and $\beta \in (1/(2N-1), 1/N)$.
Then Lebesgue a.a. $t\in\Gamma_{\beta,N}-\Gamma_{\beta,N}=[-1,1]$
have a continuum of distinct $\OPN$-codes. This gives the set
$\Gamma_{\beta,N}\cap(\Gamma_{\beta,N}+t)$ a more complicated
structure. We summarize the results in the following. In Section 2,
an algebraical and geometrical description of the set
\begin{equation*}
\UPN:=\big\{t\in [-1,1]: |\pf^{-1}(t)|=1\big\}
\end{equation*}
 (i.e., the set of $t\in[-1,1]$ having a unique $\OPN$-code) is given in Theorem \ref{th:1a},
 where throughout the paper $|A|$ denotes the number of members in
the set $A$. Section 3 is mainly devoted to investigating the
self-similar structure of $\Gamma _{\beta , N}\cap (\Gamma _{\beta ,
N}+t)$. Let
\begin{equation*}
\SPN:=\big\{t\in\UPN :\Gamma _{\beta ,
 N}\cap (\Gamma _{\beta , N}+t) \; \textrm{is a self-similar set}\big\}.
\end{equation*}
Theorem \ref{th:2a} gives a sufficient and necessary condition for
$t\in \SPN$. In Section 4, we study the set $\UPN$ for different
$\beta\in(1/(2N-1),1/N)$ culminating in Theorem \ref{th:critical
point}. Using techniques from beta-expansions,
 we obtain a \emph{critical point} $\beta_{c}\in(1/(2N-1),1/N)$
 such that $\UPN$ has positive Hausdorff dimension if $\beta\in(1/(2N-1),\beta_{c})$,
 and contains countably infinite many elements if
 $\beta\in(\beta_{c},1/N)$.
 We point out that the critical point $\beta_{c}$ is a transcendental number which is related to the famous
  Thue-Morse sequence (cf.~\cite{Komornik  Loreti}). In Section 5 we
  find the second critical point $\alpha_c=[N+1-\sqrt{(N-1)(N+3)}\,]/2\in(1/(2N-1),\beta_c)$ (see Theorem \ref{th:critical point
  S}) such that $\SPN$ has positive
 Hausdorff dimension if $\beta\in(1/(2N-1),\alpha_c)$, and contains countably infinite many
 elements if $\beta\in[\alpha_c,1/N)$.  In the following table, we give the critical points $\beta_{c}=\beta_c(N)$ and $\alpha_c=\alpha_c(N)$
  calculated for different integers $N$ by means of Mathematica.
\begin{center}
  \begin{tabular}{|c|c|c|c|c|c|c|c|c|}
  \hline
  $N$&2&3&4&5&6&7&8&9\\\hline
  $\beta_{c}\approx$ &0.39433&  0.27130 & 0.21004&  0.17221    &   0.14625&   0.12722  &  0.11265 &   0.10111
  \\\hline
 $\alpha_c\approx$&  0.38197& 0.26795 & 0.20871 & 0.17157  & 0.14590&0.12702
 &0.11252 &0.10102\\       \hline
  \end{tabular}
\end{center}
Thus for $\beta\in[\alpha_c,\beta_c)$, the set $\UPN$ (the set of
$t\in[-1,1]$ having a unique $\OPN$-code) has positive Hausdorff
dimension, but only countably many $t\in\UPN$ make the intersection
$\Gamma_{\beta,N}\cap(\Gamma_{\beta,N}+t)$ a self-similar set.

\section{Geometrical description of
$\Gamma_{\beta,N}\cap(\Gamma_{\beta,N}+t)$}

We say that the IFS $\{f_i(\cdot)\}_{i=1}^p$ satisfies the
\emph{open set condition }(OSC) if there exists a nonempty bounded
open set $O\subseteq \mathbb R$ such that $O\supseteq \bigcup
_{i=1}^p f_i(O)$, with a disjoint union on the right side. An IFS
$\{f_i(\cdot)\}_{i=1}^p$ is said to satisfy the \emph{strong separation
condition} (SSC) if the union $\Gamma=\bigcup_{i=1}^p f_i(\Gamma)$
is disjoint.

When $\beta\in(0,1/(2N-1))$ the IFS $\{\phi _d(\cdot):
d\in\OPN\}$ satisfies the SSC, so each point in
$\Gamma_{\beta,\OPN}$ has a unique $\OPN$-code. In case $\beta
=1/(2N-1)$, the IFS $\{\phi_d(\cdot):d\in\OPN\}$ fails to satisfy the
SSC but satisfies the OSC, so each point has a unique $\OPN$-code
except for countably many points having two $\OPN$-codes. However,
for the case $\beta\in(1/(2N-1), 1/N)$ the IFS $\{\phi _d(\cdot):
d\in\OPN\}$ fails to satisfy the OSC and $\Gamma_{\beta,\OPN}=[-1,1]$. In this case,
Lebesgue a.a.~$t\in[-1,1]$ have a continuum of distinct $\OPN$-codes
(cf.~\cite{Sidorov}). This gives
$\Gamma_{\beta,N}\cap(\Gamma_{\beta,N}+t)$ a more complicated
structure, since it follows (\cite{LiW2}) that for
$t\in\Gamma_{\beta,\OPN}$
\begin{equation}\label{eq:1}
\Gamma_{\beta,N}\cap(\Gamma_{\beta,N}+t)=\bigcup_{\tilde{t}}\pz\left(\prod_{\ell=1}^\infty
D_{\ell,\tilde{t}}\right)
\end{equation}
where the union is taken over all $\OPN$-codes of $t$, and for each
code $\tilde{t}=(t_\ell)_{\ell=1}^\infty \in \OPN^{\N}$
\begin{equation*}
 D_{\ell,\tilde{t}}=\ON\cap(\ON+t_\ell)=\{0,1,\dots, N-1\}
\cap(\{0,1,\dots, N-1\}+t_\ell).
\end{equation*}
Moreover, $\Gamma_{\beta,N}\cap(\Gamma_{\beta,N}+t)$  has the
following properties:

\textbf{(P1)} { the union on the right side of {\rm (\ref{eq:1})}
consists of pairwise
  disjoint sets};

\textbf{(P2)} { for each $\OPN$-code
$\tilde{t}=(t_\ell)_{\ell=1}^\infty$ of $t$, we have
\begin{equation*}
1+t-\pz\left(\prod_{\ell=1}^\infty
D_{\ell,\tilde{t}}\right)=\pz\left(\prod_{\ell=1}^\infty
  D_{\ell,\tilde{t}}\right),
\end{equation*}
  i.e., $\pz(\prod_{\ell=1}^\infty D_{\ell,\tilde{t}})$
  is centrally  symmetric. Furthermore, $1+t- \Gamma_{\beta,N}\cap(\Gamma_{\beta,N}+t)=
  \Gamma_{\beta,N}\cap(\Gamma_{\beta,N}+t)$.}

These properties can be obtained as follows. Let
$(t_\ell)_{\ell=1}^\infty$ be  a $\OPN$-code
 of $t$ and let $J=(j_\ell)_{\ell=1}^\infty \in \ON^\N$.
 If
\begin{equation*}
 \pz(J)=\sum _{\ell=1}^\infty \frac{j_\ell\beta ^{\ell-1}(1-\beta )}{N-1}\in \pz
 \left (\prod _{\ell=1}^\infty \ON \cap(\ON+t_\ell)\right
),
 \end{equation*}
 then $(j_\ell-t_\ell)_{\ell=1}^\infty \in
\ON^\N$. Note that the IFS $\{\phi_d(\cdot): d\in\ON\}$ satisfies the
SSC (since $\beta<1/N$). This implies that each point
$x\in\Gamma_{\beta,N}$ has a unique $\ON$-code. Thus
$(j_\ell-t_\ell)_{\ell=1}^\infty $ is the unique $\ON$-code of
$\pz(J)-t$, implying (P1). In addition, one can check that for each
$ \ell\ge 1$,
\begin{equation*}
 N-1+t_\ell-\ON \cap(\ON+t_\ell) = \ON
\cap(\ON+t_\ell),
\end{equation*}
 implying (P2).

Let $\Omega $ be a nonempty finite subset of $\mathbb Z$. Denote by
$\ep$ the empty word and put $\Omega^0=\{\ep\}$. For
$I\in\bigcup_{\ell=0}^\N\Omega^\ell$ and
$J\in\Omega^\N\cup\bigcup_{\ell=0}^\N\Omega^\ell$,
 let $IJ\in\Omega^\N\cup\bigcup_{\ell=0}^\N\Omega^\ell$ be the concatenation of $I$ and
 $J$. So in particular $\ep J=J$. For a nonnegative integer $k$ and a finite string $I\in\bigcup_{\ell=1}^\N\Omega^\ell$,
  let $I^{k}:=\overbrace{I\dots I}^k$ be the $k$ times
repeating of $I$ and $I^\N:=III\dots\in\Omega^\N$ be the infinite
repeating of $I$. In particular, $I^0=\ep$. For
$J=(j_\ell)_{\ell=1}^\infty\in\Omega^\N$ and $k\in\mathbb{N}$, let
$J|_k=(j_\ell)_{\ell=1}^k\in\Omega^k$. We define the algebraic
difference between two infinite strings $I=(i_\ell)_{\ell=1}^\N,
J=(j_\ell)_{\ell=1}^\N\in\Omega^\N$ by
$I-J=(i_\ell-j_\ell)_{\ell=1}^\N$, and for a positive integer $k$
let $I|_k-J|_k=(I-J)|_{k}=(i_\ell-j_\ell)_{\ell=1}^k$.

Given $\beta\in(1/(2N-1),1/N)$ and $t\in [-1,1]$, for an integer
$d\in\mathbb{Z}$, let
\begin{equation*}
\psi_d(x)=\beta
x+d(1-\beta)/(N-1)+t(1-\beta),\quad x\in\mathbb{R}.
\end{equation*}
 Then
\begin{equation*}
\Gamma_{\beta,N}+t=\bigcup_{d\in\ON}\psi_d(\Gamma_{\beta,N}+t).
\end{equation*}
For $J=(j_\ell)_{\ell=1}^k \in \ON^k$ with $k\in\mathbb{N}$, let
$\psi_J:=\psi _{j_1}\circ \cdots \circ \psi _{j_k}$ (the same for
$\phi _J$). For a real number $x$, it is easy to see that
$\psi_d(t+x)=\phi_d(x)+t$ for all $d\in\ON$. Thus by induction we
obtain
\begin{equation}\label{eq:psi-phi}
  \psi_J(t+x)=\phi_J(x)+t\quad\mbox{for all}~J\in\bigcup_{\ell=1}^\N\ON^\ell,~ x\in\mathbb{R}.
\end{equation}
The sets $\Gamma_{\beta,N}$ and $\Gamma_{\beta,N}+t$ can be
represented in a geometrical way as (cf.~\cite{Falconer})
\begin{equation*}
\Gamma_{\beta,N}=\bigcap_{k=1}^\infty\bigcup_{J\in
\ON^k}\phi_J([0,1])\quad\mbox{and}\quad
\Gamma_{\beta,N}+t=\bigcap_{k=1}^\infty\bigcup_{J\in
\ON^k}\psi_J([t,1+t]).
\end{equation*}
We call $\phi_J([0,1]),\psi_J([t,1+t])$ with $J\in \ON^k$ the\emph{
$k$-level components} of $\Gamma_{\beta,N}$ and
$\Gamma_{\beta,N}+t$, respectively. The $1$-level components of
$\Gamma_{\beta,N}$ are $\phi_0([0,1]), \phi_1([0,1]), \dots
,\phi_{N-1}([0,1])$ of length $\beta $. All gaps between them have
the same length $(1-\beta )/(N-1)-\beta $. The left endpoint of
$\phi_0([0,1])$ is $0$ and the right endpoint of $\phi_{N-1}([0,1])$
is $1$. For a $\ell$-level component $\phi_J([0,1]), J\in \ON^\ell$,
the $(\ell+1)$-level components $\phi_{J0}([0,1]), \phi_{J1}([0,1]),
\dots , \phi_{J(N-1)}([0,1])$ have the same length $\beta ^{\ell+1}$
and all gaps (called $(\ell+1)$-level gaps)
 between them have the same length $\beta
^\ell(1-\beta )/(N-1)-\beta ^{\ell+1}$. The left endpoint of
$\phi_{J0}([0,1])$ coincides with the left endpoint of
$\phi_{J}([0,1])$ and the right endpoint of $\phi_{J(N-1)}([0,1])$
coincides with the right endpoint of $\phi_{J}([0,1])$. The
requirement $\beta \in (1/(2N-1), 1/N)$ implies the following simple
properties:

 \textbf{(P3)} the length of a $k$-level gap is less than the length
of a $k$-level component, i.e.,
\begin{equation*}
 \beta
^{k-1}(1-\beta )/(N-1)-\beta ^{k}<\beta ^k;
\end{equation*}

 \textbf{(P4)} if $\phi _I([0,1])\cap \psi _J([t, t+1])\ne \emptyset
$ for  $I,J \in  \ON^k$ with $k\in\mathbb{N}$, then
\begin{equation*}
\phi _I([0,1])\cap \psi _J([t, 1+t])\cap
\Gamma_{\beta,N}\cap (\Gamma_{\beta,N}+t)\ne \emptyset.
\end{equation*}

\begin{figure}[ht]
\begin{center}
 \includegraphics[width=425pt]{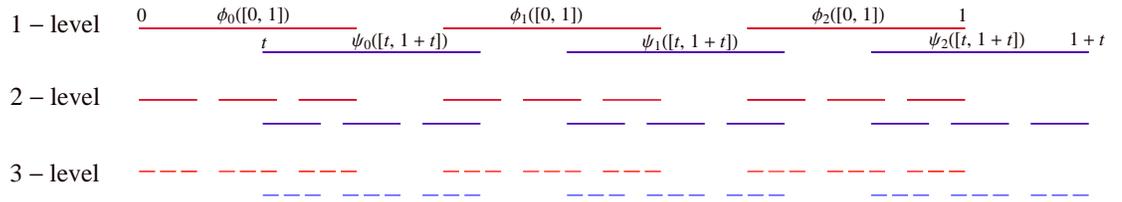}\\
\vspace{-0.4cm}\caption{{\footnotesize $N=3,~  \beta=0.28, ~t=0.19$.
The $1$-level components of $\Gamma_{\beta,N}$ are $\phi_0([0,1]),~
\phi_1([0,1])$ and $ \phi_2([0,1])$.
  The $1$-level
  components of $\Gamma_{\beta,N}+t$ are $\psi_0([t,1+t]),~ \psi_1([t,1+t])$ and $ \psi_2([t,1+t])$. Here $\mathcal{N}_t(0)=\{\psi_0([t,1+t])\}$,
  $\mathcal{N}_t(1)=\{\psi_0([t,1+t]),~ \psi_1([t,1+t])\}$ and $\mathcal{N}_t(2)=\{\psi_1([t,1+t]),~\psi_2([t,1+t])\}$.}\label{fig:1}}
  \end{center}
\end{figure}

For $J\in \ON^k$ with $k\in \mathbb N$, the \emph{neighborhood} of
$\phi_J([0,1])$ with respect to the $k-$level components of
$\Gamma_{\beta,N}+t$ is defined as (see Figure \ref{fig:1})
\begin{equation*}
\mathcal{N}_{t}(J):=\Big\{\psi_I([t,1+t]):I\in
\ON^k,~\phi_J([0,1])\cap\psi_I([t,1+t])\ne\emptyset\Big\}.
\end{equation*}
 The set
$\mathcal{N}_{t}(J)$ may be empty and $|\mathcal{N}_{t}(J)|\in
\{0,1,2\}$. For $k\ge 1$ let
\begin{equation*}
\Lambda_k:=\Big\{J\in \ON^k:|\mathcal{N}_{t}(J)|\ge
1\Big\}\quad\mbox{and}\quad\Lambda:=\Big\{J\in
\ON^\N:J|_k\in\Lambda_k \;\textrm{for all}\; k\in\mathbb {N}\Big\}.
\end{equation*}
Then  $\Gamma_{\beta,N}\cap(\Gamma_{\beta,N}+t) $ can be rewritten
in a geometrical way as
\begin{equation*}
\Gamma_{\beta,N}\cap(\Gamma_{\beta,N}+t)=\pz(\Lambda)=\bigcap_{k=1}^\infty\bigcup_{J\in\Lambda_k}\phi_J([0,1]).
\end{equation*}
A set $D\subseteq \ON$ is said to be \emph{consecutive} if
$D=\ON\cap (\ON+d)$ for some $d\in \OPN$.
\begin{proposition}\label{prop:1}
 Given $N\ge 2$ and $\beta\in(1/(2N-1),1/N)$, let $t\in [-1,1]$. If $|\mathcal{N}_{t}(J)|\le 1$ for all
  $J\in \bigcup_{\ell=1}^\N\ON^\ell$, then
\begin{equation*}
\Lambda=\prod_{\ell=1}^\infty D_\ell
\end{equation*}
  with each $D_\ell$ consecutive.
\end{proposition}
\begin{proof}
 The
condition $\beta\in(1/(2N-1),1/N)$ implies (P3), i.e., all gaps
between the intervals $\phi _d([0,1]), d\in\ON$ have the same length
strictly less than $\beta $, the length of $\phi _d([0,1])$ (see
Figure \ref{fig:1}). Thus since $t\in[-1,1]$, either
$|\mathcal{N}_t(0)|=1$ or $|\mathcal{N}_t(N-1)|=1$, which implies
that
\begin{equation*}
 D_1:=\Big\{d\in \ON:
|\mathcal{N}_{t}(d)|=1\Big\}\ne \emptyset.
\end{equation*}
It follows from $|\mathcal{N}_t(d)|\le 1$ for all $d\in\ON$ that
$D_1$ is consecutive and $\Lambda _1=D_1$.

Now for $k\in \mathbb N$ let the consecutive sets $D_1, \dots, D_k$
be chosen such that $\Lambda_k=\prod_{\ell=1}^k D_\ell$. Fix a
$J\in\Lambda_k$ and take
\begin{equation*}
D_{k+1}:=\Big\{d\in \ON:
|\mathcal{N}_{t}(J d)|=1\Big\}.
\end{equation*}
Then $D_{k+1}$ is nonempty by (P3), and is consecutive by the same
argument as above. Note that $D_{k+1}$ is independent of the choice
of $J\in\Lambda_k$. Thus $\Lambda_{k+1}=\prod_{\ell=1}^{k+1} D_\ell$
which implies $\Lambda=\prod_{\ell=1}^\infty D_\ell$ by induction.
\end{proof}
 The following theorem characterizes the set of $t\in[-1,1]$ having a unique $\OPN$-code from
a geometrical and an algebraical aspect.
\begin{theorem}\label{th:1a}
 Given $N\ge 2$ and $\beta\in (1/(2N-1),1/N)$, let $\UPN$ be the set of $t\in[-1,1]$ which have a unique
$\OPN$-code. Then the
  following conditions are equivalent.

{\rm (A)} $t\in\UPN$;

{\rm (B)} $|\mathcal{N}_{t}(J)|\le1$ for all
$J\in\bigcup_{\ell=1}^\N\ON^\ell$;

{\rm (C)} $t$ has a $\OPN$-code $(t_\ell)_{\ell=1}^\infty$ such
  that for all $k\ge 1$
  \begin{equation}\label{tunique}
\left\{
\begin{array}{lcr}
  \sum_{\ell=1}^\infty t_{k+\ell}\beta^\ell<\frac{1-N \beta}{1-\beta},& \mbox{if}&t_k<N-1\\
  \sum_{\ell=1}^\infty t_{k+\ell}\beta^\ell>-\frac{1-N \beta}{1-\beta}, &\mbox{if}& t_k>1-N.
\end{array}
\right.
\end{equation}

\end{theorem}
\begin{proof}
 ${\rm (A)}\Rightarrow{\rm (B)}$. Suppose that $|\mathcal{N}_{t}(J)|=2$
for some $J=(j_\ell)_{\ell=1}^k\in\ON^k$ with $k\ge 1$. Then either
$|\mathcal{N}_{t}(J|_{k-1}0)|=2$ or
$|\mathcal{N}_{t}(J|_{k-1}(N-1))|=2$. Without loss of generality,
let $|\mathcal{N}_{t}(J|_{k-1}0)|=2$. Then there exists $d\in \ON$
such that $|\mathcal{N}_{t}(J|_{k-1}d)|=1$ by the geometric
structure of $\Gamma_{\beta,N}\cap(\Gamma_{\beta,N}+t)$ (see Figure
\ref{fig:2}).
\begin{figure}[ht]
\centering{  \includegraphics[width=425pt]{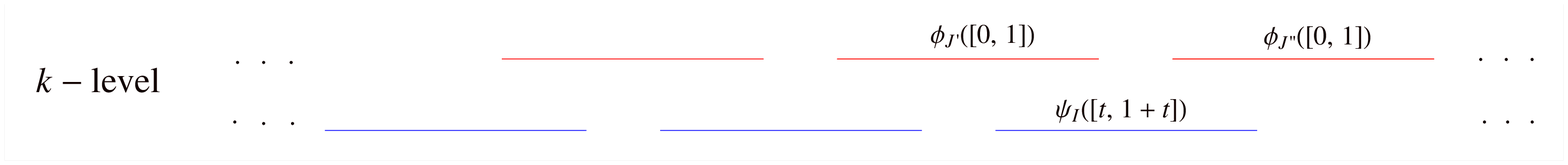}\\
\vspace{-3mm}\caption{{\footnotesize $N=3.$ Here $J'=J|_{k-1}1,
J''=J|_{k-1}2$ and
$\mathcal{N}_t(J')\cap\mathcal{N}_{t}(J'')=\{\psi_I([t,1+t])\}.$}}\label{fig:2}}
\end{figure}

Let $J'=J|_{k-1}(d-1)$ and $J''=J|_{k-1}d$. Then
\begin{equation*}
\mathcal{N}_{t}(J')\cap\mathcal{N}_{t}(J'')=\big\{\psi_{I}([t,1+t])\big\}
\end{equation*}
for some $I=i_1i_2\cdots i_{k-1}(N-1)\in \ON^k$. By (P4) we can pick
\begin{equation*}
x\in\phi_{J'}([0,1])\cap\psi_{I}([t,1+t])\cap\Gamma_{\beta,N}\cap(\Gamma_{\beta,N}+t)
\end{equation*}
and
\begin{equation*}
y\in\phi_{J''}([0,1])\cap\psi_{I}([t,1+t])\cap\Gamma_{\beta,N}\cap(\Gamma_{\beta,N}+t).
\end{equation*}
Let  $(x_\ell)_{\ell=1}^\infty$ and $(y_\ell)_{\ell=1}^\infty$ be
the unique $\ON$-code of $x$ and $y$, respectively. Then $x_k=d-1$
and $y_k=d$. On the other hand, $x-t, y-t\in \Gamma_{\beta,N}$ and
by $(x_\ell^*)_{\ell=1}^\infty, (y_\ell^*)_{\ell=1}^\infty$ we
denote their unique $\ON$-code, respectively. It follows from
(\ref{eq:psi-phi}) that
\begin{equation*}
x\in \psi_{I}([t,1+t])=\phi_{I}([0,1])+t~~\textrm{and}~~ y\in
\psi_{I}([t,1+t])=\phi_{I}([0,1])+t,
\end{equation*}
which imply $x-t, y-t\in \phi_{I}([0,1])$. Thus $x_k^*=y_k^*=N-1$.
Hence $t=x-(x-t)=y-(y-t)$ has two distinct $\OPN$-codes:
$(x_\ell-x^*_\ell)_{\ell=1}^\infty$ and
$(y_\ell-y^*_\ell)_{\ell=1}^\infty$.

 ${\rm (B)}\Rightarrow{\rm (A)}$. By Proposition \ref{prop:1},
  we have $\Gamma_{\beta,N}\cap(\Gamma_{\beta,N}+t)=\pz(\prod_{\ell=1}^\infty D_\ell)$ with $
  D_\ell$ consecutive. Thus, it follows from (\ref{eq:1}) that $t$ has a unique $\OPN$-code $(t_\ell)_{\ell=1}^\N$
  with each $t_\ell$ determined by
  $D_\ell=\ON\cap(\ON+t_\ell)$.

${\rm (B)}\Rightarrow{\rm (C)}$. It follows from Proposition
\ref{prop:1} that
$\Gamma_{\beta,N}\cap(\Gamma_{\beta,N}+t)=\pz(\prod_{\ell=1}^\infty
D_\ell)$ with each $D_\ell$ consecutive. Take
$J=(j_\ell)_{\ell=1}^\infty\in\prod_{\ell=1}^\infty D_\ell$. Then
$\pz(J)\in \Gamma _{\beta ,N}\cap (\Gamma _{\beta ,N}+t)$. Let
$J^*=(j_\ell^*)_{\ell=1}^\infty$ be the unique $\ON$-code of
$\pz(J)-t\in \Gamma _{\beta ,N}$. Thus it follows by
(\ref{eq:psi-phi}) that for each $k\ge 1$
\begin{equation*}
 \pz(J)\in \phi
_{J|_k}([0,1])\cap (\phi _{J^*|_k}([0, 1])+t)=\phi
_{J|_k}([0,1])\cap \psi _{J^*|_k}([t,1+t]),
\end{equation*}
and
\begin{equation*}
J-J^*=(j_\ell-j_\ell^*)_{\ell=1}^\infty =(t_\ell)_{\ell=1}^\infty
\end{equation*}
is the unique $\OPN$-code of $t$ (the uniqueness is given by
"$(B)\Rightarrow(A)$"). We shall prove $(t_\ell)_{\ell=1}^\infty$
satisfies (\ref{tunique}) in the following.

Case I. $t_k\ne \pm (N-1)$.

 In this case, $(j_k,j_k^*)\notin\{(N-1,0),(0,N-1)\}$. This together with the requirements in (B) imply that the
distance between the left endpoints of $\phi _{J|_k}([0,1])$ and
$\psi _{J^*|_k}([t,t+1])$ must be less than the length of the $k$-th
gap (see Figure \ref{fig:3}), i.e., $|\psi_{J^*|_k}(t)-\phi
     _{J|_k}(0)|<\beta ^{k-1}(1-\beta )/(N-1)-\beta ^k$.
\begin{figure}[ht]
 \centering{ \includegraphics[width=425pt]{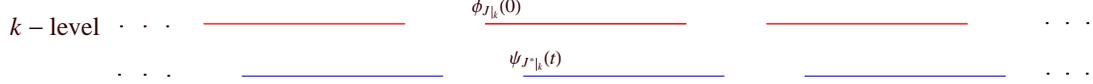}\\
 \vspace{-3mm} \caption{{\footnotesize $N=3.$ Here $\phi_{J|_k}(0)$ is the left endpoint of the $k$-level component $\phi_{J|_k}([0,1])$
 of $\Gamma_{\beta,N}$, and $\psi_{J^*|_k}(t)$ is the left endpoint of $k$-level component $\phi_{J^*|_k}([t,1+t])$ of $\Gamma_{\beta,N}+t$.}}\label{fig:3}}
\end{figure}

Thus (\ref{tunique}) follows by the following computation.
\begin{eqnarray*}
\left|\sum_{\ell=1}^\infty
     \frac{t_{k+\ell}\beta^{\ell-1}(1-\beta)}{N-1}\right|\quad &&=\beta ^{-k}\left|\sum_{\ell=k+1}^\infty
     \frac{t_{\ell}\beta^{\ell-1}(1-\beta)}{N-1}\right|=\beta ^{-k}
\left|t-\sum_{\ell=1}^k
     \frac{t_{\ell}\beta^{\ell-1}(1-\beta)}{N-1}\right|\\
     &&=\beta^{-k}\left|t-\left(\sum_{\ell=1}^k\frac{j_\ell\beta^{\ell-1}(1-\beta)}{N-1}-\sum_{\ell=1}^k\frac{j^*_\ell\beta^{\ell-1}(1-\beta)}{N-1}\right)\right|\\
     &&=\beta ^{-k}|t-(\phi_{J|_k}(0)-\phi _{J^*|_k}(0))|=\beta ^{-k}|\psi_{J^*|_k}(t)-\phi
     _{J|_k}(0)|\\
     &&<\frac{1-N\beta}{\beta(N-1)}.
    \end{eqnarray*}

 Case II. $t_k=N-1$.

In this case, $(j_k, j_k^*)=(N-1, 0)$. This together with the
requirements in (B) imply that $\phi
     _{J|_k}(0)-\psi_{J^*|_k}(t)<\beta ^{k-1}(1-\beta )/(N-1)-\beta
     ^k$. By a similar argument as in Case I, we have
\begin{equation*}
\sum_{\ell=1}^\infty
     \frac{t_{k+\ell}\beta^{\ell-1}(1-\beta)}{N-1}
    =\beta ^{-k}(\psi_{J^*|_k}(t)-\phi
     _{J|_k}(0))>-\frac{1-N\beta}{\beta(N-1)},
\end{equation*}
 leading to (\ref{tunique}).

The final case $t_k=1-N$ can be done in the same way as above.

${\rm (C)}\Rightarrow{\rm (B)}$. We will prove by induction that for
any $k\ge 1$ and $J\in \ON^k$
\begin{equation}\label{eq:7}
  \mathcal{N}_{t}(J)=\left\{
\begin{array}{ll}
 \big\{\psi_{J-(t_\ell)_{\ell=1}^k}([t,1+t])\big\},& {\rm if} ~J\in\prod_{\ell=1}^k \big(\ON\cap(\ON+t_\ell)\big)\\
 \emptyset, & {\rm otherwise}.
\end{array}
  \right.
\end{equation}
For $k=1$, let $J\in \ON\cap(\ON+t_1)$. In view of the proof of
$(B)\Rightarrow(C)$, (\ref{tunique}) becomes
\begin{equation*}
\left\{
\begin{array}{ll}
\psi_{J-t_1}(t)-\phi
     _{J}(0)<(1-\beta )/(N-1)-\beta,&~{\rm if}~t_1<N-1\\
\phi
     _{J}(0)-\psi_{J-t_1}(t)<(1-\beta )/(N-1)-\beta,& ~{\rm if}~t_1>1-N.
\end{array}
\right.
\end{equation*}
This implies (\ref{eq:7}) from the geometrical structure of $\Gamma
_{\beta , N}\cap (\Gamma _{\beta , N}+t)$.

Suppose that (\ref{eq:7}) is true for $k=n$. Let
$J=(j_\ell)_{\ell=1}^{n+1}\in \ON^{n+1}$. Then
$\mathcal{N}_{t}(J)=\emptyset$ if $J|_n\notin\prod_{\ell=1}^n
\big(\ON\cap(\ON+t_\ell)\big)$.
 Thus we assume
$J|_n\in\prod_{\ell=1}^n \big(\ON\cap(\ON+t_\ell)\big)$. For
$j_{n+1}\in \ON\cap(\ON+t_{n+1})$, (\ref{tunique}) becomes
\begin{equation*}
\left\{
\begin{array}{ll}

\psi_{J-(t_\ell)_{\ell=1}^{n+1}}(t)-\phi
     _{J}(0)<\beta ^n(1-\beta )/(N-1)-\beta ^{n+1}, &~{\rm if}~t_{n+1}<N-1\\
\phi
     _{J}(0)-\psi_{J-(t_\ell)_{\ell=1}^{n+1}}(t)<\beta ^n(1-\beta )/(N-1)-\beta ^{n+1},& ~{\rm if}~t_{n+1}>1-N,
\end{array}
\right.
\end{equation*}
 which implies (\ref{eq:7}) for $k=n+1$.
\end{proof}

 \section{The Self-similar structure of $\Gamma_{\beta,N}\cap(\Gamma_{\beta,N}+t)$}

 Let $\Omega$ be a nonempty finite subset of $\mathbb{Z}$. An
infinite string $K\in\Omega^\N$ is called \emph{strongly periodic
with period $q$} (or simply, strongly periodic) if there exist two
finite strings $I=(i_\ell)_{\ell=1}^q,
J=(j_\ell)_{\ell=1}^q\in\Omega^q$ with $q\ge 1$ such that
$K=I\,J^\N$ and $I\preccurlyeq J$, where $I\preccurlyeq J$ means
$i_\ell\le j_\ell, 1\le \ell\le q$. For two infinite strings
$I,J\in\Omega^\N$, we say $I\preccurlyeq J$ if $I|_k\preccurlyeq
J|_k$ for all $k\in\mathbb{N}$. The following lemma (cf.~\cite[Lemma
3.1]{LiW1}) gives a description of strongly periodic infinite
strings.

\begin{lemma}\label{lemma:2}
 Let $(j_\ell)_{\ell=1}^\infty\in
\ON^\N$. If there exists a
  positive integer $q$ such that $j_{\ell+q}\ge j_\ell$ for all
  $\ell\in\mathbb{N}$,
  then $(j_\ell)_{\ell=1}^\infty$ is strongly periodic with period $q$.
\end{lemma}

When $t$ has a unique $\OPN$-code $(t_\ell)_{\ell=1}^\N$, from the
proof of Theorem \ref{th:1a} it follows that there exists a sequence
of consecutive subsets $\ON\cap(\ON+t_\ell)$ such that
\begin{equation*}
\Gamma_{\beta,N}\cap(\Gamma_{\beta,N}+t)=\pz\left
(\prod_{\ell=1}^\infty \ON\cap(\ON+t_\ell)\right).
\end{equation*}
  Let $\gamma_*$ be the smallest member of
$\Gamma_{\beta,N}\cap(\Gamma_{\beta,N}+t)$. It is easy to check that
\begin{equation}\label{eq:Gamma}
\Gamma_t:=\Gamma_{\beta,N}\cap(\Gamma_{\beta,N}+t)-\gamma_*=
\pz\left(\prod_{\ell=1}^\infty\{0, \dots , N-1-|t_\ell|\}\right).
\end{equation}
Thus the Hausdorff and packing dimensions of
$\Gamma_{\beta,N}\cap(\Gamma_{\beta,N}+t)$ are given by
(cf.~\cite{LiW3})
\begin{eqnarray*}
    \dim_H\Gamma_{\beta,N}\cap(\Gamma_{\beta,N}+t)&&=\dim_H\Gamma_t=-\frac{1}{\log\beta}\underline{\lim}_{k\rightarrow
    \N}\frac{\sum_{\ell=1}^k(N-|t_\ell|)}{k};\\
     \dim_P\Gamma_{\beta,N}\cap(\Gamma_{\beta,N}+t)&&=\dim_P\Gamma_t=-\frac{1}{\log\beta}\overline{\lim}_{k\rightarrow
    \N}\frac{\sum_{\ell=1}^k(N-|t_\ell|)}{k}.
\end{eqnarray*}
The following properties make it easier to deal with $\Gamma_t$.

 \textbf{(P5)} For $I, J\in\ON^\N$, if $I\preccurlyeq J$ and $\pz(J)\in\Gamma_t$, then $\pz(I)\in\Gamma_t$;

 \textbf{(P6)} $\Gamma_t=\gamma^*-\Gamma_t$ where
$\gamma^*=\pz\big((N-1-|t_\ell|)_{\ell=1}^\infty\big)$ is the
largest member in $\Gamma_t$.

Thus, when $\Gamma_t$ is generated
  by an IFS, say $\{f_i(x)=r_i x+b_i\}_{i=1}^p$, we can
  require all $r_i>0$\,: if $r_i<0$ we can replace $f_i(x)$ by $f_i^*(x)=-r_i
  x+b_i+r_i\gamma^*.$ This follows from a simple computation (cf.~\cite{Deng, LiW1})
\begin{equation*}
f_i^*(\Gamma_t)=-r_i\Gamma_t+b_i+r_i\gamma^*=r_i(\gamma^*-\Gamma_t)+b_i=r_i\Gamma_t+b_i=f_i(\Gamma_t).
\end{equation*}
 Furthermore, we can assume $0=b_1\le b_2\le\cdots\le b_p$
since $0=\pz(0^\N)\in\Gamma_t$ by (P5).

The following theorem gives a sufficient and necessary condition for
$t\in\SPN$, i.e., the set of $t\in[-1,1]$ which have a unique
$\OPN$-code and at the same time make the intersection
$\Gamma_{\beta,N}\cap(\Gamma_{\beta,N}+t)$ a self-similar set.

\begin{theorem}\label{th:2a}
  Given $N\ge 2$ and $\beta\in(1/(2N-1),1/N)$, let $(t_\ell)_{\ell=1}^\infty$ be
   the unique $\OPN$-code of $t\in \UPN$. Then $t\in\SPN$ if and only if $(N-1-|t_\ell|)_{\ell=1}^\infty$ is strongly
  periodic.
\end{theorem}
\begin{proof}
 It suffices to prove that $\Gamma_t$, given by (\ref{eq:Gamma}), is a self-similar set if and only if
$(N-1-|t_\ell|)_{\ell=1}^\infty$ is strongly periodic. Firstly, we
prove the sufficiency. If $(N-1-|t_\ell|)_{\ell=1}^\infty\in \ON^\N$
is strongly periodic, it can be written as
$(N-1-|t_\ell|)_{\ell=1}^\infty=\sigma\,(\sigma+\tau)^\N\in \ON^\N$
where $\sigma=(\sigma_\ell)_{\ell=1}^q,
\tau=(\tau_\ell)_{\ell=1}^q\in \ON^q$ for some $q\in{\mathbb N}$ and
$\sigma+\tau=(\sigma_\ell+\tau_\ell)_{\ell=1}^q\in \ON^q$. Let
\begin{equation*}
\mathcal{S}:=\left\{\beta^{-q}\sum_{\ell=1}^{2q}\frac{j_\ell\beta^{\ell-1}(1-\beta)}{N-1}:\ON^{2q}\ni(j_\ell)_{\ell=1}^{2q}\preccurlyeq\sigma\tau\right
\}.
\end{equation*}
One can check that $\Gamma_t $ can be generated by the IFS
$\{f_s(x)=\beta ^q(x+s): s\in \mathcal{S}\}$ (cf.~\cite{LiW1}).

 Next, we will prove the necessity. By (P6),
 we can assume that $\Gamma_t$ is generated by an IFS
 $\{f_i(x)=r_ix+b_i\}_{i=1}^p$ with $r_i\in(0,1)$ and $0=b_1\le b_2\le\cdots\le
 b_p$. Note that the union
 $(0,1)=\bigcup_{q=0}^\N[\beta^{q+1},\beta^q)$ is disjoint, there
 exist some $q\ge 0$ such that $r_1\in[\beta^{q+1},\beta^q)$.

Case I. $r_1=\beta ^{q+1}$. Then for each $\ell\ge 1$, it follows
from (P5) that
\begin{equation*}
\frac{(N-1-|t_\ell|)\beta^{\ell-1}(1-\beta)}{N-1}=\pz\big(0^{\ell-1}(N-1-|t_\ell|)0^\N\big)\in\Gamma_t.
\end{equation*}
 Thus
\begin{equation*}
f_1\left(\frac{(N-1-|t_\ell|)\beta ^{\ell-1}(1-\beta
)}{N-1}\right)=\frac{(N-1-|t_\ell|)\beta ^{\ell+q}(1-\beta
)}{N-1}\in \Gamma_t
\end{equation*}
which implies that $N-1-|t_\ell|\leq N-1-|t_{\ell+q+1}|$ for each
$\ell\ge 1$. So $(N-1-|t_\ell|)_{\ell=1}^\infty$ is strongly
  periodic with period $q+1$ by Lemma \ref{lemma:2}.

Case II. $\beta ^{q+1}<r_1<\beta ^q$. Let $r_1=\beta ^{q+\gamma }$
with $0<\gamma <1$.

(IIa)  $\gamma $ is rational. Take $k\in \mathbb N$ such that
$k\gamma \in \mathbb N$. Note that the IFS
 $\{f_0(x)=r_1^kx,  f_i(x)=r_ix+b_i, 1\leq i\leq p\}$ generates
 $\Gamma_t $. Thus the conclusion can be proved in the same way as that in Case I.

(IIb) $\gamma $ is irrational. Take $k\in \mathbb N$ such that
\begin{equation}\label{kgamma}
 \beta <\beta ^{1-k\gamma
+[k\gamma ]}<\frac{1-\beta }{N-1}.
\end{equation}
This is possible since the set $\{k\gamma -[k\gamma
]:k\in\mathbb{N}\}$ is dense in the interval $(0,1)$. Let
$f_0(x)=r_1^kx$. Then for some $\beta ^{\ell-1}(1-\beta )/(N-1)\in
\Gamma_t $ we have
\begin{equation*}
 f_0\left(\frac{\beta ^{\ell-1}(1-\beta
)}{N-1}\right)=\frac{\beta ^{kq+k\gamma +\ell-1}(1-\beta
)}{N-1}<\xi:=\frac{\beta ^{kq+[k\gamma ]+\ell-1}(1-\beta )}{N-1}.
\end{equation*}

\vspace{-4.0mm}
\begin{figure}[ht]
\centering{   \includegraphics[width=425pt]{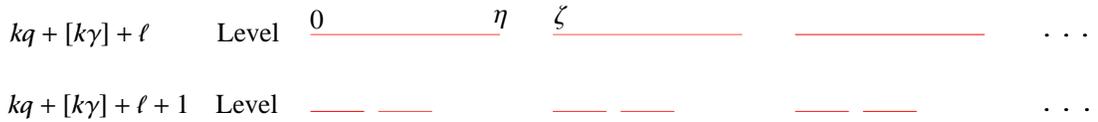}\\
 \vspace{-3mm} \caption{\footnotesize{ $\xi=(\beta ^{kq+[k\gamma
]+\ell-1}(1-\beta ))/(N-1), ~\eta=\beta ^{kq+[k\gamma ]+\ell}$. From
the geometrical construction of $\Gamma_t$, it is easy to see that
$(\eta,\xi)\cap\Gamma_t=\emptyset$.}}\label{fig:4}}
\end{figure}
 On the other hand, from (\ref{kgamma}) it follows that
\begin{equation*}
\frac{\beta ^{kq+k\gamma +\ell-1}(1-\beta )}{N-1}>\eta:=\beta
^{kq+[k\gamma ]+\ell}.
\end{equation*}
Thus $f_0(\frac{\beta ^{\ell-1}(1-\beta )}{N-1})\notin \Gamma_t$
(see Figure \ref{fig:4}), leading to a contradiction.
\end{proof}
In fact, the above proof gives a general result on the structure of
a class of subsets of the  $N$-part homogeneous Cantor set.
\begin{corollary}
Given $N\ge 2$ and $\beta \in(0, 1/N)$, let $(i_\ell)_{\ell
=1}^\infty , (j_\ell)_{\ell =1}^\infty \in \ON^\N$ satisfying
$(i_\ell)_{\ell =1}^\infty \preccurlyeq (j_\ell)_{\ell =1}^\infty$.
Then $\pz(\prod _{\ell =1}^\infty \{i_\ell, i_\ell +1,\dots , j_\ell
\})$ is a self-similar set if and only if $(j_\ell-i_\ell)_{\ell
=1}^\infty$ is strongly
  periodic.
\end{corollary}

\section{The critical point for $\UPN$}

According to a result of Sidorov~\cite[Proposition 3.8]{Sidorov}
pertaining to the general digit sets, we have that Lebesgue a.a.
$t\in[-1,1]$ have a continuum of distinct $\OPN$-codes if
$\beta\in(1/(2N-1),1/N)$. However, we will show in this section, for
the same set of $\beta$'s, that there are infinitely many
$t\in[-1,1]$ having a unique $\OPN$-code. Note that these $t$ form
exactly the set $\UPN$ defined earlier. Moreover, there is a
critical point $\beta_c\in(1/(2N-1),1/N)$ such that $\UPN$ has
positive Hausdorff dimension if $\beta\in(1/(2N-1),\beta_c)$,
 and contains countably infinite many elements if
 $\beta\in(\beta_c,1/N)$. This can be seen in Theorem \ref{th:critical
 point} which is proved by using techniques from beta-expansions.

 Given $m\ge 2$ and $\beta\in(1/m,1)$, let
  $
  \Omega_m:=\{0,1,\dots,m-1\}.
  $
    Recall that the sequence $(s_\ell)_{\ell=1}^\N\in\Omega_m^\N$ is called a
   \emph{$\beta$-expansion of $x$ with digit set $\Omega_m$}  if we can write $
  x=\sum_{\ell=1}^\infty s_\ell\beta^\ell$ with $s_\ell\in\Omega_m.$
 The largest number we can obtain in this way is
$x_{\max}:=(m-1)\beta/(1-\beta)$. Now for any $x\in(0,x_{\max}]$,
let us define a sequence $(s_\ell)_{\ell=1}^\N\in\Omega_m^\infty$
recursively by the \emph{quasi-greedy algorithm} (cf.~\cite{Vries
Komornik}): let $s_0=0$, and if $s_\ell$ is already defined for all
$\ell<n$, then let $s_n$ be the largest element in $\Omega_m$
satisfying $ \sum_{\ell=1}^n s_\ell\beta^\ell<x.$ Obviously,
$\sum_{\ell=1}^\N s_\ell\beta^\ell=x$, and we call
$(s_\ell)_{\ell=1}^\N$ the \emph{quasi-greedy $\beta$-expansion of
$x$ with digit set $\Omega_m$}. We always call
$(s_\ell)_{\ell=1}^\N$ a quasi-greedy expansion of $x$ if there is
no confusion about $\beta$ and the digit set $\Omega_m$. It is easy
to see that $(s_\ell)_{\ell=1}^\N$ is an \emph{infinite expansion}
(i.e., infinitely many $s_\ell$ are non-zeros).

We use systematically the lexicographical order between sequences:
we write $(a_\ell)_{\ell=1}^\N<(b_\ell)_{\ell=1}^\N$ or
$(b_\ell)_{\ell=1}^\N>(a_\ell)_{\ell=1}^\N$ if there exists an
$n\in\mathbb{N}$ such that $a_\ell=b_\ell$ for $\ell<n$ and
$a_n<b_n$. Furthermore, we write
$(a_\ell)_{\ell=1}^\N\le(b_\ell)_{\ell=1}^\N$ or
$(b_\ell)_{\ell=1}^\N\ge(a_\ell)_{\ell=1}^\N$ if we also allow the
equality of the two sequences. Similarly, for two $s$-blocks
$c_1\dots c_s$ and $d_1\dots d_s$, we write
$(c_\ell)_{\ell=1}^s<(d_\ell)_{\ell=1}^s$ if there exists $1\le n\le
s$ such  that $c_1\dots c_{n-1}=d_1\dots d_{n-1}$ and $c_n< d_n$.
Moreover, we write $(c_\ell)_{\ell=1}^s\le (d_\ell)_{\ell=1}^s$ if
we allow the equality of the two blocks.

Therefore, the quasi-greedy expansion of $x\in(0,x_{\max}]$ is the
largest infinite expansion among all the $\beta$-expansions of $x$
in the sense of lexicographical order. Note that $1\in(0,x_{\max}]$
since $\beta>1/m$. In the remainder of the paper we will reserve the
notation
$(\delta_\ell)_{\ell=1}^\N=(\delta_\ell(\beta))_{\ell=1}^\N$ for the
quasi-greedy $\beta$-expansion of $1$ with digit set $\Omega_m$. The
following important properties of the quasi-greedy expansion of $1$,
will be used in the proof of Theorem \ref{th:critical point}.

\begin{proposition}[Parry~\cite{Parry}]\label{prop: strictly decreasing of quasi expansion}
  Given $m\ge 2$, the map $\beta\rightarrow(\delta_\ell(\beta))_{\ell=1}^\N\in\Omega_m^\N$, with $\beta\in(1/m,1)$, is strictly decreasing
  in the sense of lexicographical order. Moreover, the map is continuous
  w.r.t. the topology in $\Omega_m^\N$ induced by the metric $d\big((a_\ell)_{\ell=1}^\N, (b_\ell)_{\ell=1}^\N\big)=2^{-\min\{j: a_j\ne
  b_j\}}$.
\end{proposition}

\begin{proposition}[de Vries and Komornik~\cite{Vries Komornik}]\label{prop: quasi greedy expansion}
  Given $m\ge 2$ and $\beta\in(1/m,1)$, let $(\gamma_\ell)_{\ell=1}^\N$ be an infinite $\beta$-expansion
  of $1$ with digit set $\Omega_m$. Then $(\gamma_\ell)_{\ell=1}^\N$ is the quasi-greedy expansion
  of $1$ if and only if for all $k\ge 1$
  \begin{equation}
  \gamma_{k+1}\gamma_{k+2}\dots\le\gamma_1\gamma_2\dots \label{eq:quasi-greedy expansion}
  \end{equation}
  in the lexicographical order.
\end{proposition}

Given $m\ge 2$, let $\overline{d}=m-1-d$ be the \emph{reflection of
the digit $d\in\Omega_m$}. For a sequence
$(a_\ell)_{\ell=1}^\N\in\Omega_m^\N$, let
$\overline{(a_\ell)_{\ell=1}^\N}=(\overline{a_\ell})_{\ell=1}^\N=(m-1-a_\ell)_{\ell=1}^\N$
be the \emph{reflection of the sequence
$(a_\ell)_{\ell=1}^\N\in\Omega_m^\N$}. A sequence
$(a_\ell)_{\ell=1}^\N\in\Omega_{m}^\N$ is said to be
\emph{admissible} if for all $k\ge 1$
\begin{equation*}
    \left\{
    \begin{array}{lcl}
      a_{k+1} a_{k+2}\dots<a_1 a_2\dots, &\mbox{if}&a_k<m-1\\
      \overline{a_{k+1}a_{k+2}\dots}<a_1 a_2\dots,&\mbox{if}& a_k>0.
    \end{array}
    \right.
  \end{equation*}
Let $(\tau_\ell)_{\ell=0}^\N\in\Omega_2^\N$ be the classical
Thue-Morse sequence, i.e., $\tau_0=0$, and if $\tau_\ell$ is already
defined for some $\ell\ge 0$, set $\tau_{2\ell}=\tau_\ell$ and
$\tau_{2\ell+1}=\overline{\tau_\ell}=1-\tau_\ell$. Then the sequence
$(\tau_\ell)_{\ell=0}^\infty$ begins as follows
\begin{equation*}
0~1101~0011~0010~1101~0010~1100~1101~0011~0010~1100\dots.
\end{equation*}
We construct a sequence
$(\lambda_\ell)_{\ell=1}^\N=(\lambda_\ell(m))_{\ell=1}^\N\in\Omega_{m}^\N$
for the even and odd numbers $m$ respectively.
\begin{equation}\label{eq:general lambda_ell}
\begin{array}{lcl}
{\rm(I)}.\quad\lambda_\ell=q-1+\tau_\ell
 ~~\mbox{for} ~~\ell\ge 1, &\mbox{if}& ~m=2q ~~\mbox{with}~~ q\ge 1;\\
{\rm(II)}.\quad\lambda_\ell=q+\tau_\ell-\tau_{\ell-1} ~~\mbox{for} ~~\ell\ge 1,
&\mbox{if}& ~m=2q+1~ ~\mbox{with} ~~q\ge 1.
\end{array}
\end{equation}
Komornik and Loreti~\cite{Komornik  Loreti} showed that
$(\lambda_\ell)_{\ell=1}^\N$ is the smallest admissible sequence in
$\Omega_{m}^\N$ in the sense of lexicographical order. Moreover,
they gave the following proposition.

\begin{proposition}[Komornik and Loreti~\cite{Komornik  Loreti}]\label{prop:small admissible}
  Let $(\lambda_\ell)_{\ell=1}^\N\in\Omega_m^\N$ be defined in \rm{(\ref{eq:general lambda_ell})}. Then for all $k\ge 1$
\begin{equation*}
  \lambda_{k+1}\lambda_{k+2} \dots <\lambda_1\lambda_2\dots,\quad\overline{\lambda_{k+1}\lambda_{k+2}\dots}<\lambda_{1}\lambda_2\dots.
\end{equation*}
\end{proposition}
 For a more general digit set $\Omega$, there also
exist some results on the smallest admissible
 sequence which is related to the Thue-Morse sequence (cf.~\cite{Allouche
 Frougny}). 

The following important theorem on the set
\begin{equation*}
\mathcal{A}_{\beta,m}:=\big\{x\in [0,x_{\max}]:~
x=\sum_{\ell=1}^\N\ep_\ell\beta^\ell,~\ep_\ell\in\Omega_m~\mbox{ has
a unique}~ \beta\mbox{-expansion}\big\}
\end{equation*}
 is due to Parry~\cite{Parry}, Erd\"{o}s et al.~\cite{Erdos}, Komornik et
al.~\cite{Komornik  Loreti} and de Vries et al.~\cite{Vries
Komornik}.
\begin{theorem}\label{th:A}
  Given $m\ge 2$ and $\beta\in(1/m,1)$, let $(\delta_\ell)_{\ell=1}^\N$ be the quasi-greedy
   $\beta$-expansion of $1$ with digit set $\Omega_m$. Then $\sum_{\ell=1}^\N \ep_\ell\beta^\ell\in\mathcal{A}_{\beta,m}$
   if and only if for all $k\ge 1$
  \begin{equation*}
    \left\{
    \begin{array}{lcl}
      \ep_{k+1}\ep_{k+2}\dots<\delta_1\delta_2\dots, &\mbox{if}&\ep_k<m-1\\
      \overline{\ep_{k+1}\ep_{k+2}\dots}<\delta_1\delta_2\dots,&\mbox{if}& \ep_k>0.
    \end{array}
    \right.
  \end{equation*}
\end{theorem}

For $m\ge 2$, let $\beta_{c,m}$ be the unique positive solution of
the following equation
 \begin{equation}\label{eq:critical point}
  1=\sum_{\ell=1}^\N \lambda_\ell\beta^\ell,
\end{equation}
where
$(\lambda_\ell)_{\ell=1}^\N=(\lambda_{\ell}(m))_{\ell=1}^\N\in\Omega_m^\N$
is defined in (\ref{eq:general lambda_ell}).
 We remark here that $\beta_{c,m}$ is a transcendental number for all $m\ge 2$ (cf.~\cite{Komornik  Loreti}).
For $m=2$, Glendinning and Sidorov~\cite{Glendinning Sidorov} have
shown that the critical point for $\mathcal{A}_{\beta,2}$ is
$\beta_{c,2}$, i.e., $\mathcal{A}_{\beta,2}$ has positive Hausdorff
dimension if $\beta<\beta_{c,2}$ and $\mathcal{A}_{\beta,2}$
contains at most countably many elements if $\beta>\beta_{c,2}$.
Their results can be generalized to the even number case, i.e., for
an even number $m\ge 2$, the critical point for
$\mathcal{A}_{\beta,m}$ is $ \beta_{c,m}$. However, it is more
intricate to find the critical point for $\mathcal{A}_{\beta,m}$ for
an odd number $m$. Inspired by \cite{Glendinning Sidorov} we show
that for an odd number $m\ge 3$, the critical point for
$\mathcal{A}_{\beta,m}$ is still $\beta_{c,m}$, the unique positive
solution of Equation $(\ref{eq:critical point})$.

Given $N\ge 2$ and $\beta\in(1/(2N-1),1/N)$, we will find the
critical point for $\UPN$, which is the set of $t\in[-1,1]$ having a
unique $\OPN$-code.

To make the connection with the theory of beta-expansions we shift
$\OPN$ to the set
\begin{equation*}
\OPN+N-1=\{0,1,\dots,2N-2\}=\OTN.
\end{equation*}
 Thus from
$[-1,1]=\pf\big(\OPN^\N\big)$ it follows that
\begin{equation*}
[0,2]=\ps\big(\OTN^\infty\big)=\left\{ \sum_{\ell=1}^\infty
\frac{\ep_\ell\beta^{\ell-1}(1-\beta)}{N-1}:
\ep_\ell\in\{0,1,\dots,2N-2\}\right\},
\end{equation*}
where $\ps:=\pi_{\Omega_{2N-1}}$ is as in (\ref{eq:pi}). Let
\begin{equation*}
\UTN:=\left\{t\in [0,2]: |\ps^{-1}(t)|=1\right\},
\end{equation*}
i.e., the set of $t\in[0,2]$ having a unique $\OTN$-code. Thus, it
is easy to see that
\begin{equation*}
 \UTN=\UPN+1.
\end{equation*}
  For $\beta\in(1/(2N-1),1/N)$,
note that
\begin{equation*}
x\in\mathcal{A}_{\beta,2N-1}\quad
\Longleftrightarrow\quad\frac{1-\beta}{\beta(N-1)}x\in\UTN.
\end{equation*}
Thus Theorem \ref{th:A} yields the the following important theorem
which could also be shown in a different way by using
(\ref{tunique}).

\begin{theorem}\label{th:tunique_shift}
  Given $N\ge 2$ and $\beta\in(1/(2N-1),1/N)$, let $(\delta_\ell)_{\ell=1}^\N$ be the quasi-greedy $\beta$-expansion of $1$ with digit set
  $\OTN$. Then $(\ep_\ell)_{\ell=1}^\N\in\ps^{-1}(\UTN)$ if and only if
for all $k\ge 1$
\begin{equation}\label{tunique_shift}
\left\{
\begin{array}{lcl}
      \ep_{k+1}\ep_{k+2}\dots<\delta_1\delta_2\dots, &\mbox{if}& \ep_k\in\{0,\dots,2N-3\}\\
      \overline{\ep_{k+1}\ep_{k+2}\dots}<\delta_1\delta_2\dots,&\mbox{if}& \ep_k\in\{1,\dots,2N-2\},
\end{array}
\right.
    \end{equation}
 where $\overline{\ep_{k+1}\ep_{k+2}\dots}$ is the
reflection of $\ep_{k+1}\ep_{k+2}\dots\in\OTN^\N$.
\end{theorem}

Therefore, dealing with the set $\UPN$ is equivalent to dealing with
the set of sequences $(\ep_\ell)_{\ell=1}^\N\in\OTN^\N$ which
satisfy (\ref{tunique_shift}). Substituting $m=2N-1$ in
(\ref{eq:general lambda_ell}), we get the smallest admissible
sequence $(\lambda_\ell)_{\ell=1}^\infty\in\OTN^\N$ which starts
with
\begin{equation*}
 N(N-1)(N-2)N\quad (N-2)(N-1)N(N-1)\quad (N-2)(N-1)N(N-2)\dots.
\end{equation*}
 It is helpful to give another equivalent definition of the sequence $(\lambda_\ell)_{\ell=1}^\infty\in\OTN^\N$
  (cf.~\cite{Komornik  Loreti}), i.e.,
\begin{equation}\label{lambda sequence}
\begin{array}{l}
 \lambda_1=N,\quad \lambda_{2^{n+1}}=\overline{\lambda_{2^n}}+1=2N-1-\lambda_{2^n}\quad\mbox{for}~ n=0,1,\dots,\\
\lambda_{2^n+\ell}=\overline{\lambda_\ell}=2N-2-\lambda_\ell \quad\quad\quad\mbox{for}~ 1\le \ell<2^n,~ n=1,2,\dots.
\end{array}
\end{equation}
 So it is easy to see $\lambda_{2^n}=N$ for $n=0,2,4,\dots$ and
$\lambda_{2^n}=N-1$ for $n=1,3,5,\dots$.

\begin{theorem}\label{th:critical point}
Given $N\ge 2, ~\beta\in(1/(2N-1),1/N)$, let $\UPN$ be the set of
$t\in[-1,1]$ having a unique $\OPN$-code and
$\beta_c\in(1/(2N-1),1/N)$ be the unique positive solution of
Equation \rm{(\ref{eq:critical point})} with
$(\lambda_\ell)_{\ell=1}^\N\in\OTN^\N$ defined in \rm{(\ref{lambda
sequence})}. Then

 \rm{(1)} If $\beta\in(1/(2N-1),\beta_c)$, then $\dim_H \UPN>0$;

 \rm{(2)} If $\beta=\beta_c$, then $|\mathcal{U}_{\beta_c,\pm N}|=2^{\aleph_0}$
 and $\dim_H\mathcal{U}_{\beta_c,\pm N}=0$;

\rm{(3)} If $\beta\in(\beta_c,1/N)$, then $|\UPN|=\aleph_0$.
\end{theorem}
Since $\UPN=\UTN-1$, the critical point of $\UPN$ is equal to the
critical point of $\UTN$. Thus we only need to show the
corresponding conclusions for the set $\UTN$.

Using Proposition \ref{prop: quasi greedy expansion} and Proposition
\ref{prop:small admissible}, we obtain
$(\delta_\ell(\beta_c))_{\ell=1}^\N=(\lambda_\ell)_{\ell=1}^\N$,
i.e., $(\lambda_\ell)_{\ell=1}^\N$ is the quasi-greedy
$\beta_c$-expansion of $1$ with digit set $\OTN$. The proof of
Theorem \ref{th:critical point} will be divided into several lemmas.

\begin{lemma}\label{Lemma:6}
 $\lambda_k\dots\lambda_{k+2^n-2}<\lambda_1\dots\lambda_{2^n-1}$ for
 any $n\ge 2$ and any $k\in\{2,\dots,2^n-1\}$;
 $\overline{\lambda_k\dots\lambda_{k+2^n-2}}<\lambda_1\dots\lambda_{2^n-1}$
 for any $n\ge 2$ and any $k\in\{1,\dots,2^n-1\}$.
\end{lemma}
\begin{proof}
Since for $n=2$ the lemma is quickly checked, let $n\ge 3$ and
$k\in\{2,\dots,2^n-1\}$. Then by Proposition \ref{prop:small
admissible} $\lambda_k\lambda_{k+1}\dots<\lambda_1\lambda_2\dots$,
which implies
$\lambda_k\dots\lambda_{k+2^n-2}\le\lambda_1\dots\lambda_{2^n-1}$.
It is easy to check that
$\lambda_k\dots\lambda_{k+2^n-2}<\lambda_1\dots\lambda_{2^n-1}$ for
$k< 7$. For all other $k$ we can write $k=2^s+2^p+j$ with $1\le p<
s<n$ and $1\le j<2^p$. It follows from \cite[Lemma 5.4]{Komornik
Loreti} that
\begin{equation*}
  \lambda_k\dots\lambda_{k+2^{p+1}-j}<\lambda_{j}\dots\lambda_{2^{p+1}}\le\lambda_1\dots\lambda_{2^{p+1}-j+1}
\end{equation*}
  which implies
  $\lambda_k\dots\lambda_{k+2^n-2}<\lambda_1\dots\lambda_{2^n-1}
  $, since
  $n>p+1$.

 For the second inequality, ignoring the trivial cases $k=1$ and $2$, suppose $k=2^q+j$ with $1\le
j<2^q$ and $1\le q<n$. Then it again follows from \cite[Lemma
5.5]{Komornik Loreti} that
\begin{equation*}
\overline{\lambda_k\dots\lambda_{k+2^q-j}}<\lambda_{j}\dots\lambda_{2^q}\le\lambda_1\dots\lambda_{2^q-j+1}.
\end{equation*}
which implies that
$\overline{\lambda_k\dots\lambda_{k+2^n-2}}<\lambda_1\dots\lambda_{2^n-1}$,
since $n>q$.
\end{proof}

\begin{lemma}\label{lemma:7}
Let $n\ge 3$ be an odd integer. If
 $\overline{\lambda_k\dots\lambda_{2^n-1}}=\lambda_1\dots\lambda_{2^n-k}$
 for some $k\in\{1,\dots,2^n-1\}$, then $\lambda_{2^n-k+1}=N$.
  \end{lemma}
\begin{proof}
 Suppose
$\overline{\lambda_k\dots\lambda_{2^n-1}}=\lambda_1\dots\lambda_{2^n-k}$.
It can not happen that $ k<2^{n-1}$ since then we will obtain that
$\overline{\lambda_k\dots\lambda_{k+2^{n-1}-2}}=\lambda_1\dots\lambda_{2^{n-1}-1}$
which contradicts Lemma \ref{Lemma:6}. It is also impossible that
 $k=2^{n-1}$ since then
$N-2=\overline{\lambda_{2^{n-1}}}=\lambda_1=N$. Thus we must have
$k>2^{n-1}$. From the definition of $(\lambda_\ell)_{\ell=0}^\N$ in
(\ref{lambda sequence}) it follows that
 \begin{equation*}
\lambda_{k-2^{n-1}}\dots\lambda_{2^{n-1}-1}=\overline{\lambda_k\dots\lambda_{2^n-1}}=\lambda_1\dots\lambda_{2^n-k},
\end{equation*}
which implies $N\ge\lambda_{2^n-k+1}\ge \lambda_{2^{n-1}}=N$ by
Proposition \ref{prop:small admissible}.
\end{proof}

We want to approximate $(\lambda_\ell)_{\ell=1}^\N$ by eventually
periodic sequences which satisfy (\ref{eq:quasi-greedy expansion}).
This does not work for the obvious choice
$(\lambda_1\dots\lambda_{2^n})^\N$. Thus we define for $n\ge 0$
\begin{equation*}
 C_n^\N=\lambda_1\dots\lambda_{2^n}(\lambda_{2^n+1}\dots\lambda_{2^{n+1}})^\N.
\end{equation*}
Since for all $n\ge 0$ we have
 $\lambda_{2^{n+1}}>\overline{\lambda_{2^n}}$, we obtain
 that
\begin{equation*}
\lambda_1\dots\lambda_{2^n}(\lambda_{2^n+1}\dots\lambda_{2^{n+1}})^3
 >\lambda_1\dots\lambda_{2^{n+1}}\lambda_{2^{n+1}+1}\dots\lambda_{2^{n+2}},
\end{equation*}
 which implies

 \textbf{(P7)}
$C_0^\N>C_1^\N>\dots>C_{n}^\N>\dots>(\lambda_\ell)_{\ell=1}^\N$ in
the lexicographical order.

\begin{lemma}\label{lemma:7'}
  Let $n\ge 3$ be an odd number. Then for any $k\ge 1$ we have
  $\sigma^k(C_{n}^\N)<C_{n}^\N$, where $\sigma$ is the left-shift
  map.
\end{lemma}
\begin{proof}
Since $C_n^\N$ is an eventually periodic sequence in $\OTN^\N$, we
only have to check the lemma for $k\in\{1,\dots,2^{n+1}-1\}$. For
$k=2^n-1$ or $2^{n+1}-1$, it is easy to check that
$\sigma^{k}(C_n^\N)<C_n^\N$. Then we only need to consider the
following two cases.

 (I) $k\in\{1,\dots,2^n-2\}$.
It follows from Lemma \ref{Lemma:6} that
\begin{equation*}
\sigma^k(C_n^\N)=\lambda_{k+1}\dots\lambda_{2^n+k-1}\dots<\lambda_1\dots\lambda_{2^n-1}\lambda_{2^n}(\lambda_{2^n+1}\dots\lambda_{2^{n+1}})^\N=C_n^\N.
\end{equation*}

(II) $k\in\{2^n,\dots,2^{n+1}-2\}$. Write $k=2^n+\ell$. Then, by the
definition of $(\lambda_\ell)_{\ell=1}^\N$,
\begin{eqnarray*}
\sigma^k(C_n^\N)&=&\lambda_{k+1}\dots\lambda_{2^{n+1}-1}\lambda_{2^{n+1}}(\lambda_{2^n+1}\dots\lambda_{2^{n+1}})^\N\\
&=&\overline{\lambda_{\ell+1}\dots\lambda_{2^n-1}}\lambda_{2^{n+1}}(\lambda_{2^n+1}\dots\lambda_{2^{n+1}})^\N.
\end{eqnarray*}
If
$\overline{\lambda_{\ell+1}\dots\lambda_{2^n-1}}<\lambda_1\dots\lambda_{2^n-\ell-1}$,
we have shown that $\sigma^k(C_n^\N)<C_n^\N$. Otherwise, $\ell\ge 2$
and we have by Proposition \ref{prop:small admissible} that
$\overline{\lambda_{\ell+1}\dots\lambda_{2^n-1}}=\lambda_1\dots\lambda_{2^n-\ell-1}$.
Using Lemma \ref{lemma:7} we obtain that also
$\lambda_{2^{n+1}}=N=\lambda_{2^n-\ell}$. Thus it is enough to show
\begin{equation*}
\lambda_{2^n+1}\dots\lambda_{2^{n+1}-1}<\lambda_{2^n-\ell+1}\dots\lambda_{2^{n+1}-\ell-1}.
\end{equation*}
 Taking reflections on both sides, this is equivalent to
showing
$\lambda_1\dots\lambda_{2^n-1}>\overline{\lambda_{2^n-\ell+1}\dots\lambda_{2^{n+1}-\ell-1}}$,
which is true by Lemma \ref{Lemma:6} since $\ell\ge 2$.
\end{proof}

\begin{lemma}\label{lemma:8}
Let $n\ge 3$ be an odd integer and $
\xi_n=(N-1)\lambda_1\dots\lambda_{2^n-1},
\eta_n=(N-2)\lambda_1\dots\lambda_{2^n-1}. $ Then for any $k\in
\{0,\dots,2^n-1\}$
\begin{eqnarray*}
    &&\sigma^k(\xi_n\eta_n)<\lambda_1\dots\lambda_{2^{n+1}-k},\quad
    \sigma^k(\overline{\xi_n\eta_n})<\lambda_1\dots\lambda_{2^{n+1}-k},\quad\sigma^k(\eta_n\overline{\xi_n})\le\lambda_1\dots\lambda_{2^{n+1}-k},\\
   &&\sigma^k(\overline{\eta_n}\xi_n)<\lambda_1\dots\lambda_{2^{n+1}-k},\quad
   \sigma^k(\xi_n\overline{\xi_n})\le\lambda_1\dots\lambda_{2^{n+1}-k},\quad
\sigma^k(\overline{\xi_n}\xi_n)<\lambda_1\dots\lambda_{2^{n+1}-k}.
\end{eqnarray*}
\end{lemma}
\begin{proof}
  Since the lemma is quickly checked for $k=0$ and $1$, we can assume
  $k\in\{2,\dots,2^n-1\}$. It follows by $\lambda_{2^n}=N-1$ (since $n$ is odd) that
 \begin{equation*}
\sigma^k(\xi_n\eta_n)=\lambda_k\dots\lambda_{2^n-1}(N-2)\lambda_1\dots\lambda_{2^n-1}
  <\lambda_k\dots\lambda_{2^n-1}\lambda_{2^n}\dots\lambda_{2^{n+1}-1}\le\lambda_1\dots\lambda_{2^{n+1}-k}.
\end{equation*}
  For the second inequality, note that
  $\sigma^k(\overline{\xi_n\eta_n})=\overline{\lambda_k\dots\lambda_{2^n-1}}N
  \overline{\lambda_1\dots\lambda_{2^n-1}}.$
   If~
  $\overline{\lambda_k\dots\lambda_{2^n-1}}<\lambda_1\dots\lambda_{2^n-k}$,
  we have shown $\sigma^k(\overline{\xi_n\eta_n})<\lambda_1\dots\lambda_{2^{n+1}-k}$. Otherwise, it follows by Proposition \ref{prop:small admissible}
  that $\overline{\lambda_k\dots\lambda_{2^n-1}}=\lambda_1\dots\lambda_{2^n-k}$ which implies $k>2$.
  Thus we obtain by
  Lemma \ref{lemma:7} that $\lambda_{2^n-k+1}=N$. Hence we only have to show
$\overline{\lambda_1\dots\lambda_{2^n-1}}<\lambda_{2^n-k+2}\dots\lambda_{2^{n+1}-k}$
which is equivalent to showing
$\lambda_1\dots\lambda_{2^n-1}>\overline{\lambda_{2^n-k+2}\dots\lambda_{2^{n+1}-k}}$.
This is true by Lemma \ref{Lemma:6} since $k>2$. Therefore,
$\sigma^k(\overline{\xi_n\eta_n})<\lambda_1\dots\lambda_{2^{n+1}-k}$
for $k\in\{2,\dots,2^n-1\}$. The remaining four inequalities follow
from Lemma \ref{Lemma:6} and the fact that for
$k\in\{2,\dots,2^n-1\}$
\begin{eqnarray*}
&&\sigma^k(\eta_n\overline{\xi_n})=\sigma^k(\xi_n\overline{\xi_n})=\lambda_k\dots\lambda_{2^n-1}\overline{(N-1)\lambda_{1}\dots\lambda_{2^n-1}}
=\lambda_k\dots\lambda_{2^{n+1}-1},\\
&&\sigma^k(\overline{\eta_n}\xi_n)=\sigma^k(\overline{\xi_n}\xi_n)=\overline{\lambda_k\dots\lambda_{2^n-1}}(N-1)\lambda_{1}\dots\lambda_{2^n-1}
=\overline{\lambda_k\dots\lambda_{2^{n+1}-1}}.
\end{eqnarray*}
\end{proof}
From Lemma \ref{lemma:7'} and Proposition \ref{prop: quasi greedy
expansion} it follows that $C_n^\N$ is the quasi-greedy expansion of
$1$ for some base $\beta_n$, i.e.,
$(\delta_\ell(\beta_n))_{\ell=1}^\N=C_n^\N$. Then we obtain from
(P7) and Proposition \ref{prop: strictly decreasing of quasi
expansion} that $\beta_n$ increases to $\beta_c$ as
$n\rightarrow\N$. Thus for $\beta<\beta_c$ there exists a large odd
number $n\ge 3$
 such that $\beta<\beta_n<\beta_c$, which together with Proposition
\ref{prop: strictly decreasing of quasi expansion} imply that
\begin{equation*}
(\delta_\ell(\beta))_{\ell=1}^\N>(\delta_\ell(\beta_n))_{\ell=1}^\N=C_n^\N=\lambda_1\dots\lambda_{2^n}(\lambda_{2^n+1}\dots\lambda_{2^{n+1}})^\N.
\end{equation*}
 It follows from
Lemma \ref{lemma:8} and Theorem \ref{th:tunique_shift} that
\begin{equation*}
X_A^{(n)}\subseteq\ps^{-1}(\mathcal{U}_{\beta,2N-1}),
\end{equation*} where $X_A^{(n)}$ is a subshift of finite type $
X^{(n)}_A:=\big\{(e_\ell)_{\ell=1}^\N\in\mathfrak{A}^\N:~
A({e_\ell,e_{\ell+1}})=1\big\} $ over the alphabet
$\mathfrak{A}=\{\xi_n,\eta_n,\overline{\xi_n},\overline{\eta_n}\}$
 defined by the matrix
\begin{equation*}
A=\left(\begin{array}{cccc} 0&1&1&0\\
0&0&1&0\\
1&0&0&1\\
1&0&0&0
\end{array}\right).
\end{equation*}
It is easy to obtain that $r(A)$, the spectral radius of $A$, equals
$\frac{1+\sqrt{5}}{2}$. Since $\ps(X_A^{(n)})$ is a graph-directed
set satisfying the OSC for large $n$, we conclude from \cite{Mauldin
Williams} that
\begin{equation*}
\dim_H\mathcal{U}_{\beta,2N-1}\ge\dim_H\ps(X_A^{(n)})=\frac{\log
r(A)}{-2^n\log\beta}=
\frac{\log\frac{1+\sqrt{5}}{2}}{-2^n\log\beta}>0,
\end{equation*}
which establishes Part (1) of Theorem \ref{th:critical point}.

In the following we will show Part (2) and (3) simultaneously. Let
\begin{equation*}
w_n:=\lambda_1\dots\lambda_{2^n}.
\end{equation*}
Then by the definition of $(\lambda_\ell)_{\ell=1}^\N$ in
(\ref{lambda sequence}) it is easy to check that
$w_n\overline{w_n}<w_{n+1}$, which implies

\textbf{(P8)} $
(w_0\overline{w_0})^\N<(w_1\overline{w_1})^\N<\dots<(w_n\overline{w_n})^\N<\dots<(\lambda_\ell)_{\ell=1}^\N
$ in the lexicographical order.

\begin{lemma}\label{Lemma:next 2^n terms}
   Given $N\ge2, ~\beta\ge\beta_{c}$ and $(\ep_\ell)_{\ell=1}^\N\in\ps^{-1}(\UTN)$, if $\ep_k<2N-2$ and
   $\ep_{k+1}\cdots\ep_{k+2^n}=w_n$ for some $k, n\ge 0$,
    then $\ep_{k+1}\dots\ep_{k+2^{n+1}}=w_n\overline{w_n}$ or
$w_{n+1}$.
    Similarly, if $\ep_k>0$ and  $\ep_{k+1}\cdots\ep_{k+2^n}=\overline{w_n}$ for some $k, n\ge 0$,
     then $\ep_{k+1}\dots\ep_{k+2^{n+1}}=\overline{w_n}w_n$ or
$\overline{w_{n+1}}$.
\end{lemma}
\begin{proof}
 Let $(\delta_\ell)_{\ell=1}^\N:=(\delta_\ell(\beta))_{\ell=1}^\N$. It follows from $\beta\ge\beta_c$ and
 Proposition \ref{prop: strictly decreasing of quasi expansion} that
\begin{equation*}
(\delta_\ell)_{\ell=1}^\N\le
(\delta_\ell(\beta_c))_{\ell=1}^\N=(\lambda_\ell)_{\ell=1}^\N .
\end{equation*}
Using (\ref{tunique_shift}) and the assumption $\ep_k<2N-2$, we
obtain that $
\ep_{k+1}\dots\ep_{k+2^{n+1}}\le\delta_1\dots\delta_{2^{n+1}}\le\lambda_1\dots\lambda_{2^{n+1}}.
$ Note that
$\ep_{k+1}\cdots\ep_{k+2^n}=w_n=\lambda_1\dots\lambda_{2^n}$, then
$\ep_{k+2^n+1}\dots
\ep_{k+2^{n+1}}\le\lambda_{2^n+1}\cdots\lambda_{2^{n+1}}.$
  On the other hand, from (\ref{tunique_shift}) and the fact
  $\ep_{k+2^n}=\lambda_{2^n}>0$ it follows that
   $\overline{\ep_{k+2^n+1}\dots\ep_{k+2^{n+1}}}\le\delta_1\dots\delta_{2^n}\le\lambda_1\dots\lambda_{2^n}.$
Thus by the definition of $(\lambda_\ell)_{\ell=1}^\N$ in
(\ref{lambda sequence}), we obtain
\begin{equation*} \lambda_{2^n+1}\cdots\lambda_{2^{n+1}-1}\lambda_{2^{n+1}}\ge\ep_{k+2^n+1}\dots\ep_{k+2^{n+1}}\ge\overline{\lambda_{1}\dots\lambda_{2^{n}}}=\lambda_{2^n+1}\dots\lambda_{2^{n+1}-1}(\lambda_{2^{n+1}}-1),
\end{equation*}
 which implies $\ep_{k+1}\cdots\ep_{k+2^{n+1}}=w_n\overline{w_n}$ or
$w_{n+1}$.

  The result for $\ep_k>0$ and $\ep_{k+1}\dots\ep_{k+2^n}=\overline{\lambda_1\dots\lambda_{2^n}}$ follows similarly.
\end{proof}

\begin{lemma}\label{lemma:eventually periodic}
  Let $N\ge 2$ and $\beta\in(\beta_c,1/N)$. Then there exists some integer $n^*=n^*(\beta)\ge 0$
  such that $\ps^{-1}(\UTN\setminus\{0,2\})$ contains only eventually periodic sequences,
   either with period $1$ and period block $N-1$ or with period $2^{n+1}$ and period block $w_n\overline{w_n}$ for some $n\le n^*$.
\end{lemma}

\begin{proof}
For $\beta\in(\beta_c,1/N)$, let
$(\delta_\ell)_{\ell=1}^\N:=(\delta_{\ell}(\beta))_{\ell=1}^\N$. The
proof will be split into two cases: Case I treats
$(\delta_\ell)_{\ell=1}^\N>(w_0\overline{w_0})^\N$, and Case II
treats $(\delta_\ell)_{\ell=1}^\N\le (w_0\overline{w_0})^\N$.

Fix a
sequence $(\ep_\ell)_{\ell=1}^\N\in\ps^{-1}(\UTN)$. In terms of
Theorem \rm{\ref{th:tunique_shift}}, it is easy to see that
$\ps^{-1}(\UTN)$
 is \emph{reflection invariant}, i.e.,
 it contains $(\ep_\ell)_{\ell=1}^\N$
 if and only if it contains $(\overline{\ep_\ell})_{\ell=1}^\N=(2N-2-\ep_\ell)_{\ell=1}^\N$. Note that $\overline{N-1}=N-1$ and
 that the existence of a period block $\overline{w_n}w_n$ implies the existence of a period block $w_n\overline{w_n}$.
  So we can assume by reflection that $\ep_1\in\{0,\dots,N-1\}$. Ignoring
the trivial case $(\ep_\ell)_{\ell=1}^\N=0^\N$, let $j\ge 1$ be the
least integer such that $\ep_j>0$. By Proposition \ref{prop:
strictly decreasing of quasi expansion}, it follows from
$\beta_c<\beta<1/N$ that
\begin{equation*}
(N-1)^\N=(\delta_\ell(1/N))_{\ell=1}^\N<(\delta_\ell)_{\ell=1}^\N<
(\delta_\ell(\beta_c))_{\ell=1}^\N=(\lambda_\ell)_{\ell=1}^\N,
\end{equation*}
which together with (\ref{tunique_shift}) imply
$\ep_j\in\{1,\dots,N\}$. Moreover, we obtain from this with
(\ref{tunique_shift}) that
\begin{equation*}
\ep_{j+1}\ep_{j+2}\dots\in\prod_1^\N\{N-2,N-1,N\}.
\end{equation*}

 {\it Case I.
$(w_0\overline{w_0})^\N<(\delta_\ell)_{\ell=1}^\N<(\lambda_\ell)_{\ell=1}^\N$.}

It then follows from (P8) that there exists an integer $n^*\ge 0$
such that
$(w_{n^*}\overline{w_{n^*}})^\N<(\delta_\ell)_{\ell=1}^\N\le(w_{n^*+1}\overline{w_{n^*+1}})^\N$.

(Ia) $\ep_j\in\{1,\dots,N-1\}$. One case is that
$\ep_{j+1}\ep_{j+2}\dots=(N-1)^\N$, otherwise, let first $s\ge j$ be
the least integer such that
$\ep_{s+1}\in\{N,N-2\}=\{w_0,\overline{w_0}\}$, and then let
$p=p(s)\ge 0$ be the largest integer such that
  $\ep_{s+1}\dots\ep_{s+2^p}=w_p$ or $\overline{w_p}$. Note that when $s>j$, then $0< \ep_s=N-1<2N-2$ or when $s=j$, then
  $0<1\le\ep_s\le N-1<2N-2$. Thus substituting $k=s$
and $n=p$ in Lemma \ref{Lemma:next 2^n terms} we obtain
$\ep_{s+1}\dots\ep_{s+2^{p+1}}\in\{w_p\overline{w_p},\overline{w_p}w_p,
w_{p+1}, \overline{w_{p+1}}\}$.

If $\ep_{s+1}\dots\ep_{s+2^{p+1}}=w_{p+1}$ or $\overline{w_{p+1}}$,
substituting $k=s$ and $n=p+1$ in Lemma \ref{Lemma:next 2^n terms},
we can determine the next $2^{p+1}$ terms as above. Otherwise, using
that $\ep_{s+2^p}=\lambda_{2^p}$ or $\overline{\lambda_{2^p}}$, and
then substituting $k=s+2^p$ and $n=p$ in Lemma \ref{Lemma:next 2^n
terms} we can determine the next $2^p$ terms. This procedure can be
continued.

Note that $\ep_{s+1}\ep_{s+2}\dots$ can not have block $w_{n^*+1}$,
otherwise, it follows from (P8) that for some $\ell\ge s$, either
\begin{equation*}
\ep_{\ell+1}\ep_{\ell+2}\dots\ge
(w_{n^*+1}\overline{w_{n^*+1}})^\N\ge(\delta_\ell)_{\ell=1}^\N
\end{equation*} with $\ep_\ell<N\le 2N-2$, or
\begin{equation*}
 \overline{\ep_{\ell+1}\ep_{\ell+2}\dots}\ge
(w_{n^*+1}\overline{w_{n^*+1}})^\N\ge(\delta_\ell)_{\ell=1}^\N
\end{equation*}
with $\ep_\ell>N-2\ge0$. This is in contradiction with
(\ref{tunique_shift}).

Therefore, $(\ep_\ell)_{\ell=1}^\N$ must be eventually periodic
either with period block $N-1$ or with period block
$w_n\overline{w_n}$ for some $n\le n^*$.

(Ib) $\ep_j=N$. Let $s=j-1$ in (Ia) and then the result follows by
the same argument.

\emph{Case II. $(N-1)^\N<(\delta_\ell)_{\ell=1}^\N\le
(w_0\overline{w_0})^\N$.}

We conclude in this case that $\ep_{j+1}\ep_{j+2}\dots=(N-1)^\N$.
Otherwise, there exists a $s\ge j$ such that $\ep_{s+1}=w_0$ or
$\overline{w_0}$. Thus by the same argument as in Case I, we obtain
for some integer $\ell\ge s$ that either $
\ep_{\ell+1}\ep_{\ell+2}\dots\ge(w_0\overline{w_0})^\N\ge(\delta_\ell)_{\ell=1}^\N$
with $ \ep_\ell<2N-2$, or
$\overline{\ep_{\ell+1}\ep_{\ell+2}\dots}\ge(w_0\overline{w_0})^\N\ge(\delta_\ell)_{\ell=1}^\N$
with $\ep_\ell>0$,
 leading to a contradiction with (\ref{tunique_shift}).
\end{proof}

 Lemma \ref{lemma:eventually periodic} yields Part (3) of Theorem
\ref{th:critical point} directly. Let $\mathcal{G}$ be the set of
sequences in $\OTN^\N$ which are eventually periodic with period
block $N-1$ or $w_n\overline{w_n}$ for some integer $n\ge 0$. Then
the set $\mathcal{G}$ is countable. When $\beta=\beta_c$, it follows
from Lemma \ref{Lemma:next 2^n terms} and the proof of Lemma
\ref{lemma:eventually periodic} that
 $
 \ps^{-1}(\mathcal{U}_{\beta_c,2N-1}\setminus\{0,2\})\setminus\mathcal{G}
 $
 is included in the set of sequences of the form
\begin{equation*}
\tau(w_0\overline{w_0})^{k_0}(w_0\overline{w_{i_1'}})^{k_0'}(w_{i_1}\overline{w_{i_1}})^{k_1}(w_{i_1}\overline{w_{i_2'}})^{k_1'}
 \dots(w_{i_n}\overline{w_{i_n}})^{k_n}(w_{i_n}\overline{w_{i_{n+1}'}})^{k_n'}\dots,
\end{equation*}
 where $\tau\in\bigcup_{k=0}^\N\OTN^k, ~k_n\in\mathbb{N}\cup\{0\},~k_n'\in\{0,1\}$ and $0<i_1'\le i_1< i_2'\le i_2<\dots \le i_n<i_{n+1}'\le i_{n+1}<\dots$,
 together with their reflections. Thus, since the length of the block $w_n$ is growing exponentially, $\dim_H\mathcal{U}_{\beta_c,2N-1}=0$ (cf.~\cite{Froberg, Glendinning Sidorov}). Note that $\ps^{-1}(\mathcal{U}_{\beta_c,2N-1})$ contains the
set of sequences of the form
\begin{equation*}
(w_0\overline{w_0})^{k_0}\dots(w_n\overline{w_n})^{k_n}\dots,\quad
k_n\in\mathbb{N},
\end{equation*}
 and the fact that
$w_n\overline{w_n}$ can not be written as concatenation of two or
more blocks of the form $w_\ell\overline{w_\ell}$ with $\ell<n$.
Therefore, $|\mathcal{U}_{\beta_c,2N-1}|=2^{\aleph_0}$ which yields
Part (2), and so finishes the proof of Theorem \ref{th:critical
point}.

\section{The critical point for $\SPN$}
In this section we show that there exist infinitely many $t\in\SPN$,
i.e., there exist infinitely many $t\in[-1,1]$ having a unique
$\OTN$-code and making the intersection
$\Gamma_{\beta,N}\cap(\Gamma_{\beta,N}+t)$ a self-similar set.
Moreover, we find the critical point $\alpha_c$ for $\SPN$, i.e.,
the set $\SPN$ has positive Hausdorff dimension if
$\beta\in(1/(2N-1),\alpha_c)$, and contains countably infinite many
elements if $\beta\in[\alpha_c, 1/N)$. We are able to prove that
$\alpha_c$ is strictly smaller than $\beta_c$, the critical point of
$\UPN$ which is the set of $t\in[-1,1]$ having a unique $\OPN$-code.

In order to using techniques from beta-expansions, we consider the
set $\STN=\SPN+1$. Thus it follows from Theorem \ref{th:2a} that for
$\beta\in(1/(2N-1),1/N)$,
\begin{equation*}
\STN=\left\{\ps((\ep_\ell)_{\ell=1}^\N)\in\UTN:(N-1-|\ep_\ell-N+1|)_{\ell=1}^\N~\mbox{is
strongly periodic}\right\}.
\end{equation*}
Let $\Psi$ be a map from $\OTN$ to $\ON$ defined by
\begin{equation*}
\Psi(\ep)=N-1-|\ep-N+1|,
\end{equation*} then $\Psi$ induces a map
on blocks (for $\xi=\xi_1\dots\xi_k\in\OTN^k$ we let
$\Psi(\xi)=\Psi(\xi_1)\dots\Psi(\xi_k)$), and a map $\Psi_\N:
~\OTN^\N~\rightarrow~\ON^\N$ given by
$\Psi_\N((\ep_\ell)_{\ell=1}^\N)=(\Psi(\ep_\ell))_{\ell=1}^\N$. Then
$\STN$ can be rewritten as
\begin{equation}\label{eq:STN}
\STN=\UTN\cap\ps\big(\bigcup\limits_{\textbf{c}}\Psi_\N^{-1}(\textbf{c})\big),
\end{equation}
where the union is taken over all strongly periodic sequences
$\textbf{c}=(c_\ell)_{\ell=1}^\N\in\ON^\N$.
\begin{theorem}\label{th:critical point S}
  Given $N\ge 2$ and $\beta\in(1/(2N-1),1/N)$, let $\Gamma_{\beta,N}$ be the $N$-part homogeneous Cantor set,
   and $\SPN$ be the set of $t\in[-1,1]$ having a unique
  $\OPN$-code and making the intersection $\Gamma_{\beta,N}\cap(\Gamma_{\beta,N}+t)$ a self-similar set. Denote
   $\alpha_c:=[N+1-\sqrt{(N-1)(N+3)}\,]/2$. Then

   \rm{(1)} If $\beta\in(1/(2N-1),\alpha_c)$, $\dim_H\SPN>0$;

    \rm{(2)} If $\beta\in[\alpha_c,1/N)$, $|\SPN|=\aleph_0$.
\end{theorem}
Since $\STN=\SPN+1$, we only need to consider the corresponding
conclusions for $\STN$. A simple computation yields that $\alpha_c$
satisfies the equation
\begin{equation*}
1=N \alpha_c+\sum_{j=2}^\N (N-1)\alpha_c^j.
\end{equation*}
Then it follows by Proposition \ref{prop: quasi greedy expansion}
that $
(\delta_\ell(\alpha_c))_{\ell=1}^\N=N(N-1)^\N=\lambda_1\lambda_2^\N
$ is the quasi-greedy $\alpha_c$-expansion of $1$. It follows from
Proposition \ref{prop: strictly decreasing of quasi expansion} and
\begin{equation*}
(\delta_\ell(\alpha_c))_{\ell=1}^\N=\lambda_1\lambda_2^\N>(\lambda_\ell)_{
\ell=1}^\N=(\delta_\ell(\beta_c))_{\ell=1}^\N
\end{equation*} that
$\alpha_c<\beta_c$. The proof of Theorem \ref{th:critical point S}
will be divided into several lemmas.

\begin{lemma}\label{lemma:5_1}
  Given $N\ge 2$ and $n\in\mathbb{N}$, let $\alpha_n$ be defined by
  $(\delta_\ell(\alpha_n))_{\ell=1}^\N=(N(N-1)^{n-1})^\N$. If
  $\beta<\alpha_n$, then $\dim_H\STN>0$.
\end{lemma}
\begin{proof}
  Let $v_n=N(N-1)^{n-1}$ and
  $\overline{v_n}=(N-2)(N-1)^{n-1}$ be its reflection. It follows from $\beta<\alpha_n$ and Proposition \ref{prop: strictly decreasing of quasi expansion}
  that $(\delta_\ell(\beta))_{\ell=1}^\N>(\delta_\ell(\alpha_n))_{\ell=1}^\N=(N(N-1)^{n-1})^\N,$ which implies that for any $k\in\{0,1,\dots,n-1\}$
 \begin{eqnarray*}
      &&\sigma^k(v_n v_n)\le \delta_1(\alpha_n)\dots\delta_{2n-k}(\alpha_n),\quad
      \sigma^k(\overline{v_n} v_n)< \delta_1(\alpha_n)\dots\delta_{2n-k}(\alpha_n),\\
&&\sigma^k(v_n
\overline{v_n})<\delta_1(\alpha_n)\dots\delta_{2n-k}(\alpha_n),\quad
\sigma^k(\overline{v_n v_n})<
\delta_1(\alpha_n)\dots\delta_{2n-k}(\alpha_n).
  \end{eqnarray*}
  Thus by Theorem \ref{th:tunique_shift} we obtain that
\begin{equation*}
\prod_1^\N\{v_n,\overline{v_n}\}\subseteq\ps^{-1}(\mathcal{U}_{\beta,2N-1}).
\end{equation*}
Since $\Psi(v_n)=(N-2)(N-1)^{n-1}=\Psi(\overline{v_n})$, it is easy
to see that
\begin{equation*}
\prod_1^\N\{v_n,\overline{v_n}\}\subseteq\Psi^{-1}_\N\big(((N-2)(N-1)^{n-1})^\N\big).
\end{equation*}
Thus noting that $((N-2)(N-1)^{n-1})^\N$ is obviously a strongly
periodic sequence in $\ON^\N$, it follows from (\ref{eq:STN}) that
\begin{equation*}
\prod_1^\N\{v_n,\overline{v_n}\}\subseteq\ps^{-1}(\UTN)\cap\Psi^{-1}_\N\big(((N-2)(N-1)^{n-1})^\N\big)\subseteq\ps^{-1}(\STN)
\end{equation*}
which implies
$\dim_H\STN\ge\dim_H\ps(\prod_1^\N\{v_n,\overline{v_n}\})>0$.
\end{proof}
Since $(\delta_\ell(\alpha_n))_{\ell=1}^\N=(N (N-1)^{n-1})^\N$
decreases to $N(N-1)^\N=(\delta_\ell(\alpha_c))_{\ell=1}^\N$ in the
sense of lexicographical order as $n\rightarrow \N$, we obtain from
Proposition \ref{prop: strictly decreasing of quasi expansion} that
$\alpha_n$ increases to $\alpha_c$. Thus for each $\beta<\alpha_c$,
there exists some $n\in\mathbb{N}$ such that $\beta<\alpha_n$ and
then $\dim_H\STN>0$ by Lemma \ref{lemma:5_1}.
This finishes the proof of Part (1) of
Theorem \ref{th:critical point S}.

In the following we will show Part (2). For
$\beta\in[\alpha_c,1/N)$, it follows by Proposition \ref{prop:
strictly decreasing of quasi expansion} that
$(\delta_\ell(\beta))_{\ell=1}^\N\le(\delta_\ell(\alpha_c))_{\ell=1}^\N=N(N-1)^\N$,
which together with Theorem \ref{th:tunique_shift} imply the
following property:

\textbf{(P9)} For $N\ge 2$ and $\beta\in[\alpha_c,1/N)$, any block
in $\mathcal{F}$ is forbidden in $\ps^{-1}(\UTN)$ where
\begin{equation*}
\mathcal{F}=\bigcup_{k=0}^\N\bigcup_{\tau=N-2}^{N-1}\{\tau N(N-1)^k
N, ~\overline{\tau}(N-2)(N-1)^k(N-2)\}.
\end{equation*}
For a positive integer $n$, let $\mathcal{B}_n$ be the set of blocks
of length $n$ occurring in elements of $\ps^{-1}(\UTN)$, i.e.,
\begin{equation*}
\mathcal{B}_n:=\big\{\ep_{i+1}\ep_{i+2}\dots\ep_{i+n}:~ i\ge
0,~~(\ep_\ell)_{\ell=1}^\N\in\ps^{-1}(\UTN)\big\}.
\end{equation*}

\begin{lemma}\label{lemma:5_3}
  Given $N\ge 2$ and $\beta\in[\alpha_c,1/N)$, let~ ${\rm{\bf b}}=b_1\dots b_p\in\{N-2,N-1\}^p$ with $b_1=N-1$ for some $p\in\mathbb{N}$.
  Then $\Psi^{-1}({\rm{\bf b}})\cap\mathcal{B}_p=\{(N-1)^p\}$ or $\{\xi,\overline{\xi}\}$ for some $\xi\in\{N-2,N-1,N\}^p$.
\end{lemma}
\begin{proof}
 Let
$\xi=\xi_1\dots \xi_p\in\Psi^{-1}(\textbf{b})\cap\mathcal{B}_p$.
Then it follows from $\textbf{b}\in\{N-2,N-1\}^p$ and the definition
of $\Psi$ that $\xi\in\{N-2,N-1,N\}^p$. Note that
$\Psi^{-1}(N-1)=\{N-1\}$ and $\Psi^{-1}(N-2)=\{N-2,N\}$.

(I) $\textbf{b}=(N-1)^p$. Then
$\Psi^{-1}(\textbf{b})\cap\mathcal{B}_p=\{(N-1)^p\}.$

(II) $\textbf{b}\ne(N-1)^p$. Let $b_{k_1}=b_{k_2}=\dots=b_{k_s}=N-2$
for $1<k_1<k_2<\dots<k_s\le p$, and $b_k=N-1$ for $k\ne k_i$. Then
also $\xi_k=N-1$ for $k\ne k_i$.
 Moreover, if $\xi_{k_1}=N$, then it
follows from (P9) that $\xi_{k_2}=N-2, ~\xi_{k_3}=N, ~\xi_{k_4}=N-2$
and so on. Similarly, if $\xi_{k_1}=N-2$ we will obtain by (P9) that
$\xi_{k_2}=N, ~\xi_{k_3}=N-2,~\xi_{k_4}=N$ and so on. Thus,
$\Psi^{-1}(\textbf{b})\cap\mathcal{B}_p=\{\xi,\overline{\xi}\}$.
\end{proof}

\begin{lemma}\label{lemma:5_4}
  Given $N\ge 2$ and $\beta\in[\alpha_c,1/N)$, let
  $\textbf{c}=(c_\ell)_{\ell=1}^\N\in\ON^\N$ be a strongly periodic sequence.
  Then $\ps^{-1}(\UTN)\cap\Psi_\N^{-1}(\textbf{c})$
  is at most countable.
\end{lemma}
\begin{proof}
Note by $\beta\ge \alpha_c$ that
$(\delta_\ell(\beta))_{\ell=1}^\N\le(\delta_\ell(\alpha_c))_{\ell=1}^\N=N(N-1)^\N$.
Thus for any sequence $(\ep_\ell)_{\ell=1}^\N\in\ps^{-1}(\UTN)$,
 we obtain by the same argument as in Lemma \ref{lemma:eventually periodic} that
  \begin{equation*}
\ep_k\ep_{k+1}\dots\in\prod_1^\N\{N-2,N-1,N\}
  \end{equation*}
  for some large $k\in\mathbb{N}$, which implies that
  $\Psi_\N(\ep_k\ep_{k+1}\dots)\in\{N-2,N-1\}^\N$. Let $\textbf{c}=a_1\dots a_q(b_1\dots
  b_q)^\N$ with $a_\ell\le b_\ell,~1\le \ell\le q$ be a strongly
  periodic sequence in $\ON^\N$ such that
  $\ps^{-1}(\UTN)\cap\Psi^{-1}_\N(\textbf{c})\ne\emptyset$. Then
\begin{equation*}
b_1\dots b_q\in\{N-2,N-1\}^q.
\end{equation*}

  Case I. $b_1\dots b_q=(N-2)^q$. It follows from
  $\Psi^{-1}(N-2)=\{N-2,N\}$ that $\Psi_\N^{-1}(\textbf{c})\subseteq\Psi^{-1}(a_1\dots
  a_q)\{N-2,N\}^\N$. Note by (P9) (with $\tau=N-2, k=0$) that blocks $N(N-2)^2$ and
  $(N-2)N^2$ are forbidden in $\ps^{-1}(\UTN)$. Thus
\begin{equation*}
\ps^{-1}(\UTN)\cap\Psi^{-1}_\N(\textbf{c})\subseteq\Psi^{-1}(a_1\dots
a_q)\{N^\N,(N(N-2))^\N,((N-2)N)^\N,(N-2)^\N\}
\end{equation*}
which is at most countable.

 Case II. $b_1\dots b_q\ne (N-2)^q$. Then there exists  $b_k=N-1$ for some $k\in\{1,\dots,q\}$.
 Note that
\begin{equation*}
\textbf{c}=a_1\dots a_q(b_1\dots b_q)^\N=a_1\dots a_q b_1\dots
b_{k-1}(b_k\dots b_q b_1\dots b_{k-1})^\N.
\end{equation*}
 It follows from Lemma \ref{lemma:5_3} that there exists a $q$-block
 $\xi=\xi_1\dots\xi_q\in\{N-2,N-1,N\}^q$ such that $\Psi^{-1}(b_k\dots b_q b_1\dots
 b_{k-1})\cap\mathcal{B}_q=\{\xi,\overline{\xi}\}$. Thus
\begin{equation*}
\ps^{-1}(\UTN)\cap\Psi_\N^{-1}(\textbf{c})\subseteq\ps^{-1}(\UTN)\cap\Big(\Psi^{-1}(a_1\dots
a_q b_1\dots b_{k-1})\prod_1^\N\{\xi,\overline{\xi}\}\Big).
\end{equation*}
  Note that since $\Psi^{-1}(\textbf{c})$ and $\ps^{-1}(\UTN)$ are all reflection invariant, $\ps^{-1}(\UTN)\cap\Psi_\N^{-1}(\textbf{c})$
  is also reflection invariant. Thus we only need to consider the
  following three cases.

(IIa) $\xi=(N-1)^q$. Then $\prod_1^\N\{\xi,\overline{\xi}\}$
collapses to a single point $(N-1)^\N$.

  (IIb) $\xi=(N-1)^{\ell}N\xi_{\ell+2}\dots\xi_{q-r-1} N(N-1)^{r}$ with $\ell\ge 1, r\ge
  0$ and $\ell+r\le q-1$ (note that $\xi=(N-1)^{\ell}N(N-1)^{r}$ if $\ell+r=q-1$). It follows by (P9) that
  blocks $\xi\xi$ and $\overline{\xi}\overline{\xi}$ are forbidden
  in $\ps^{-1}(\UTN)\cap\Psi_\N^{-1}(\textbf{c})$. Thus $\prod_1^\N\{\xi,\overline{\xi}\}$ collapses to two points
  $(\xi\overline{\xi})^\N$ and $(\overline{\xi}\xi)^\N$.

  (IIc) $\xi=(N-1)^{\ell}N\xi_{\ell+2}\dots\xi_{q-r-1}(N-2)(N-1)^{r}$ with $\ell\ge 1, r\ge
  0$ and $\ell+r\le q-2$.  By the same argument as in (IIb) we also obtain
  that $\ps^{-1}(\UTN)\cap\Psi^{-1}(\textbf{c})$ is at most countable.
 \end{proof}
It follows from Lemma \ref{lemma:5_4} and (\ref{eq:STN}) that for
$\beta\in[\alpha_c,1/N)$, the set
\begin{equation*}
\ps^{-1}(\STN)=\ps^{-1}(\UTN)\cap\bigcup_{\textbf{c}}\Psi_\N^{-1}(\textbf{c})
=\bigcup_{\textbf{c}}\left(\ps^{-1}(\UTN)\cap\Psi_\N^{-1}(\textbf{c})\right)
\end{equation*}
is at most countable since the union on the right is countable.
Note that for $\beta\in[\alpha_c,1/N)$, $\{0^q (N-1)^\N:\!
q\in\mathbb{N}\}\subseteq\ps^{-1}(\STN).$ This gives Part (2),
finishing the proof of Theorem \ref{th:critical point S}.

\section{Final remarks}
In this paper we determined the size of two types of sets $\UPN$,
and $\SPN$, where $\UPN$ is the set of $t\in\Gamma_{\beta,N}-\Gamma_{\beta,N}$ having a unique
$\OPN$-code and $\SPN$ is the set of $t$ not only having a unique
code but also making the intersection
$\Gamma_{\beta,N}\cap(\Gamma_{\beta,N}+t)$ a self-similar set. It
follows from \cite{Sidorov} that  for $\beta\in(1/(2N-1),1/N)$ there
also exist a lot of $t\in\Gamma_{\beta,N}-\Gamma_{\beta,N}=[-1,1]$
having exactly $p$ different $\OPN$-codes for any integer $p\ge 2$.
Let
$$
\mathcal{F}_{\beta,\pm
N}^{(p)}:=\{t\in\Gamma_{\beta,N}-\Gamma_{\beta,N}: t ~\textrm{has
exactly}~p~\textrm{different}~ \OPN\textrm{-codes}\},
$$
and
$$
\mathcal{S}_{\beta,\pm N}^{(p)}:=\{t\in\mathcal{F}_{\beta,\pm N}^{(p)}: \Gamma_{\beta,N}\cap(\Gamma_{\beta,N}+t)~\textrm{ is a self-similar set}\}.
$$

\textbf{Problem}. How large is the set $\mathcal{F}_{\beta,\pm
N}^{(p)}$ for a given positive integer $p\ge 2$? How to characterize
this set? This is also an open problem for beta-expansions. Moreover, how large is the set $\mathcal{S}_{\beta,\pm N}^{(p)}$?

\end{document}